\documentclass[11pt]{article}
\usepackage{tikz}
\usepackage{setspace}
\usepackage{amsthm,amsfonts,amssymb,epsfig,graphics,amsmath,amsbsy,subfigure}%,showkeys}
\usepackage{color}		% Need the color package
\usepackage{epsfig}
\newcommand{\intL}{\int\limits}
\newcommand{\half}{^\infty_0 }
\newcommand{\intR}{\int\limits_{\mathbb{R}} }
\newcommand{\intRR}{\int\limits_{\mathbb{R}^2} }
\newcommand{\RR}{\mathbb{R}}
\newcommand{\xx}{\mathbf{x}}
\newcommand{\zz}{\mathbf{z}}
\newcommand{\pp}{\mathbf{p}}
\usepackage{amssymb}
\usepackage{amsmath}
\usepackage{graphicx}
\newtheorem{theorem}{Theorem}
\newtheorem{definition}[theorem]{Definition}
\newtheorem{remark}[theorem]{Remark}
\newtheorem{corollary}[theorem]{Corollary}
\newtheorem{lemma}[theorem]{Lemma}

\title{On the determination of a function from cylindrical Radon transforms}
\author{ Sunghwan Moon}
\date{Department of Mathematical Sciences\\[-0.1em]
\normalsize
Ulsan National Institute of Science and Technology\\[-0.1em]
\normalsize
Ulsan 689-798, Republic of Korea\\[-0.1em]
\normalsize
{\tt shmoon@unist.ac.kr}}
\begin{document}
%\title{On the determination of a function from cylindrical Radon transforms}

\maketitle

\begin{abstract}
This paper is devoted to Radon-type transforms arising in Photoacoustic Tomography that uses integrating line detectors. 
We consider two situations: when the line detectors are tangent to the boundary of a cylindrical domain and when the line detectors are located on a plane. 
We present the analogue of the Fourier slice theorems for each case of the Radon-type transforms. 
Also, we provide several new inversion formulas, a support theorem, and stability estimate and necessary range condition results. 

%\begin{keywords} cylindrical, Radon transform, Thermoacoustic, tomography
%\end{keywords}
%\begin{classcode}44A12; 92C55 \end{classcode}\bigskip

\end{abstract}

\section{Introduction}
%%%%%%%%%%%%%%%%%%%%%%%%
Photoacoustic Tomography (PAT) is the best-known example of a hybrid imaging method. It has applications to functional brain imaging of animals, early cancer diagnostics, and imaging of vasculature~\cite{haltmeier09}. In 1880, A.G. Bell discovered the photo-acoustic effect~\cite{bell80}. 
Nearly 100 years later, it was realized that this effect enables one to combine advantages of pure optical and ultrasound imaging, providing both high optical contrast and ultrasonic resolution~\cite{bowen81}. 
Nevertheless, PAT has rather low cost. %compared with some conventional imaging techniques.In spite of its rather low cost

In PAT, one induces an acoustic pressure wave inside of an object of interest by delivering optical energy~\cite{kuchmentk08,xuw06}. 
{\color{black} This acoustic wave is measured by ultrasound detectors placed on outside of the object.
%Consider that a biological object is irradiated by a wide, homogeneous, but extremely short electro-magnetic pulse in radio-frequency range.
Irradiated cancerous cells absorb several times more electromagnetic (EM) energy than the surrounding healthy tissues. 
%Hence, when the internal photoacoustic sources are reconstructed form the measurements to image, such cells are displayed with high contrast.
Thus, the absorption function, the density of energy absorbed at a location, contains valuable diagnostic information.
%The initially generated acoustic wave contains diagnostic information including classification of healthy cells and cancer cells.
The photoacoustic effect implies that the initial value of a pressure wave is essentially the absorption function~\cite{kuchment12}.}
%The initial value of the pressure wave can be used to determine the values of the absorption function.}
Mathematically, in the model we study, the problem boils down to recovering the initial data of the three dimensional wave equation from the values of the solution observed at all times on the surface. 
{\color{black} This idea was implemented in the middle of 1990s \cite{krugerlfa95,krugerrk99,oraevskyejt96}.
(There are some surveys and books for details and further references, e.g.,~\cite{krugerkmrrkh00,kuchmentk08,kuchment12,kuchment14book,patchs07,xuw06}.)}

Various types of detectors have been considered for measuring the acoustic data: point-like detectors, line detectors, planar detectors, cylindrical detectors, and circular detectors (see \cite{burgholzerbmghp07,burgholzerhphs05,grattpnp11,moontat14,zangerls10,zangerlsh09,zangerls09}).
While point-like detectors approximately measure the pressure at a given point, other types of detectors are integrating. %The former one, the classic approach, brings out spherical Radon transform. The position of detectors corresponds to the center of integration sphere, and the propagation distance of the acoustic wave to the detector correspond to its radius. 
{\color{black} More specifically, the line detector measures the value of integral of the pressure along its length. 
%We obtain this integration value at all the time. 
 This data is equivalent to an Abel-type transform of the surface integral over the cylinders with central axis corresponding to a detector line and whose radii are arbitrary (for detailed information, see~\cite{haltmeier09,paltaufnhb07}).
Since an Abel-type transform can be inverted, PAT with the line detectors leads to the mathematical problem of reconstructing a function from integrals over cylindrical surfaces.}

Various configurations of line detectors have been considered in~\cite{burgholzerbmghp07,burgholzerhphs05,haltmeierf06,haltmeier09,haltmeier11,paltaufnhb07}.
In this article, we deal with two basic geometries: the line detectors are tangent to a cylinder, and the line detectors are located on a plane. 
We call these \textbf{the cylindrical version} and \textbf{the planar version}, respectively.
Some inversion formulas for the first version were found in~\cite{haltmeier11}. 
In this text, we address other issues of importance in tomography~\cite{natterer01,nattererw01}: a support theorem, a stability estimate, and necessary range conditions. We also consider an $n$-dimensional case of this model. 
%We derive new inversion formulas as well.
In the planar version, Haltmeier \cite{haltmeier09} provided a two-step procedure for image reconstruction.
In this text, we define a cylindrical Radon transform and present an analogue of the Fourier slice theorem as well as a stability estimate and necessary range conditions.

%However, given some Radon type transform, one is usually interested not only in inversion formulas, but also in uniqueness, stability estimates, and a range description~\cite{natterer01,nattererw01}. These are the issues we address below.

Two different versions of the cylindrical Radon transform are discussed in sections~\ref{defiandworkcylinder} and~\ref{plane}. Various inversion formulas of the cylindrical version of a cylindrical Radon transform different from those in~\cite{haltmeier11} are provided in section~\ref{recon}. Section~\ref{sec:uniquess} is devoted to a support theorem for this version of the transform. In sections~\ref{estiandrangecylinder} and~\ref{rangecylinder}, we provide a stability estimate and the necessary range conditions of the transform.
We also provide inversion formulas, a stability estimate, and the necessary range conditions of the plane version of a cylindrical Radon transform in sections~\ref{reconplane},~\ref{estiandrangeplane}, and~\ref{rangeplane}, respectively. 
In sections~\ref{sec:reconnd} and~\ref{planendimension}, we study $n$-dimensional cases of cylindrical Radon transforms.% and describe the range, respectively.
%%%%%%%%%%%%%%%%%%%%%%%%%%%%%%%%%%%%%%%%%%%
\section{Cylindrical geometry}\label{defiandworkcylinder}
%%%%%%%%%%%%%%%%%%%%%%%%%%%%%%%%%%%%%%%%%%%%%%%%%
We explain first the mathematical model arising in PAT with line detectors as introduced in~\cite{haltmeier11}.
Let $B^k_R$ be the ball in $\RR^k$ centered at the origin with radius $R>0$.
Then $B^2_R\times\RR$ is the solid cylinder in $\RR^3$ with radius $R$. 
%Let $u:S^1\rightarrow [R,\infty)$ be a continuous function.
For fixed $p\in\RR$ and $\boldsymbol\theta\in S^1$, let 
$$
L_C(\boldsymbol\theta,p)=\{(x,y,z)\in\RR^3:(x,y)\cdot\boldsymbol\theta=R,z=p\}
$$
be the line occupied by a detector.
Detector lines $L_C(\boldsymbol\theta,p)$ are tangent to the cylinder $B^2_R\times\RR$ (see Figure 1).

\begin{figure}[here]
\begin{center}
  \begin{tikzpicture}[>=stealth,scale=1]
    \draw[dotted] (0,-2.3) -- (0,4);
    \draw[<->,loosely dashed] (0,-1.5) -- (2,-1.5);
    \fill[semitransparent,green] (0,0) ellipse (0.4 and 0.8) ;
    \draw[very thick] (0,1.9) ellipse (2 and 0.5) ;
    \draw[very thick, dashed] (2,-1.5) arc (0:180:2 and 0.5) ;
    \draw[very thick] (-2,-1.5) arc (180:360:2 and 0.5) ;
    \draw[very thick] (-2,-1.5) -- (-2,1.9);
    \draw[very thick] (2,-1.5) -- (2,1.9);
    \draw[densely dashed] (-3.8,-1.8) -- (1.5,3.5);%L_C(\boldsymbol\theta,p)
    \draw[very thick] (-3.5,-0.5) -- (0.5,3.5);
    \draw[very thick] (-2.5,-1.5) -- (1.5,2.5);
    \draw[very thick,rotate around={-45:(-3,-1)}] (-3,-1) ellipse (1.4141/2 and 0.35) ;
    \draw[very thick,rotate around={-45:(0.5,3.5)},dashed] (0.5,3.5) arc (180:360:1.4141/2 and 0.35) ;
    \draw[very thick,rotate around={-45:(1.5,2.5)}] (1.5,2.5) arc (0:180:1.4141/2 and 0.35) ;
    \node at (0,-2.8) {(a)};
    \node at (-4.1,-2.2) {$L_C(\boldsymbol\theta,p)$};
    \node at (0.1,4.2) {$z$};
    \node at (0.5,-1.3) {R};
    \node at (-2.2,0) {$p$};
    \node at (0,0) {$f$};
    \draw[dotted] (5,-2.2) -- (5,4);    
    \draw[densely dashed] (7,-2.2) -- (7,4);    
    \draw[<->,loosely dashed] (7,-1.8) -- (5,-1.8);    
    \draw[dotted] (7,.5) -- (5,.5); 
    \draw[very thick] (7,0.5) circle (1.5);
    \draw[<->,loosely dashed] (8.5,0.5) -- (7,0.5);    
    \draw[densely dashed] (7+1.732*1.5/2,0.5+1.5/2) -- (7,0.5);    
    \draw[very thick,rotate around={-45:(7.4,.5)}] (7.4,.5) arc (50:80:0.4 and 0.4) ;
    \node at (7,-2.8) {(b)};
    \node at (5.1,4.2) {$z$};
    \node at (6,-1.5) {R};
    \node at (4.8,.5) {$p$};
    \node at (7.5,0.2) {$r$};
    \node at (7.7,0.7) {$\psi$};
  \end{tikzpicture}
%        \subfigure[]{
%            \includegraphics[width=0.37\textwidth]{cylindercopy.jpg}        }\qquad\qquad%
%          %\subfigure[]{
%           %  \includegraphics[width=0.3\textwidth]{cylinderzplane.eps}}
%                       \subfigure[]{
%            \includegraphics[width=0.3\textwidth]{cylindertheta.eps}
%             }
\caption{(a) the cylinder of integration whose the central axis is tangent to the cylinder $B^2_R\times \RR$ %the integral of $f$ supported in the cylinder $B^2_R\times\RR$ over cylindrical surfaces %(b) $z=z_0$ plane, 
and (b) the restriction to the $\{(t\boldsymbol\theta,z):t\in\RR,z\in\RR\}$ plane}
\end{center}
\label{fig:cylinder}
\end{figure}
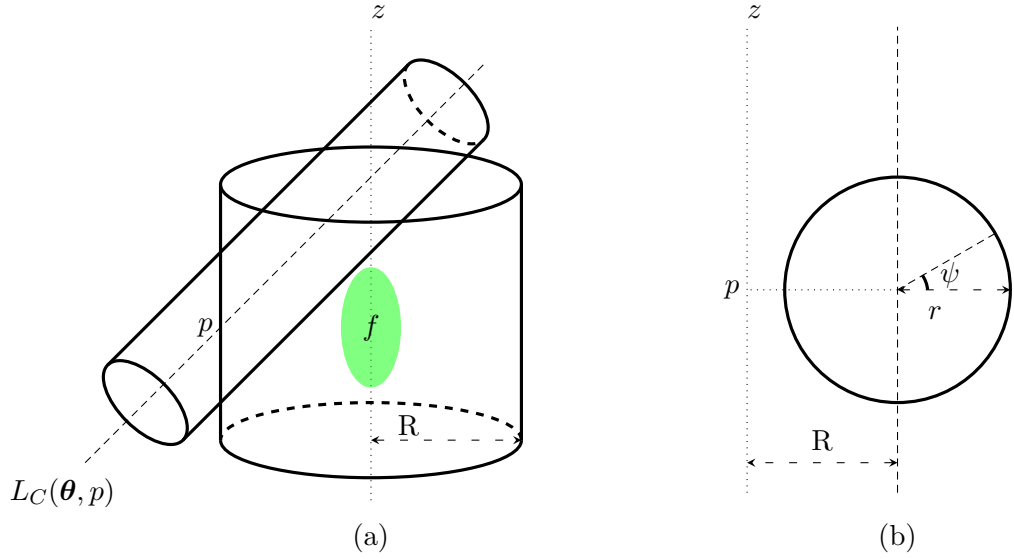 
\begin{definition}
The cylindrical Radon transform $R_C$ maps a function $f\in C^\infty_c(B^2_R\times\RR)$ to 
 $$
 R_Cf(\boldsymbol\theta,p,r)=\displaystyle\frac{1}{2\pi r}\iint\limits_{d(L_C(\boldsymbol\theta,p),(x,y,z))=r} f(x,y,z)d\varpi,
 $$
 for $(\boldsymbol\theta,p,r)\in S^1\times\RR\times[0,\infty)$.
Here $d\varpi$ is the area measure on the cylinder 
$$
\{(x,y,z)\in\RR^3:d(L_C(\boldsymbol\theta,p),(x,y,z))=r\}
$$ 
and 
$$
d(L_C(\boldsymbol\theta,p),(x,y,z)):=\sqrt{(R-(x,y)\cdot\boldsymbol\theta)^2+(p-z)^2}
$$ 
denotes the Euclidean distance between the line $L_C(\boldsymbol\theta,p)$ and the point $(x,y,z)$.
\end{definition}

%\begin{remark}
%When one fixes $\boldsymbol\theta$ and restricts the cylindrical Radon transform $R_Cf$ to a plane $\{(t\boldsymbol\theta,z):t\in\RR,z\in\RR\}$, $R_Cf$ turns into the 2-dimensional circular Radon transform whose centers are located at $(R \boldsymbol\theta,z)$ (see Figure 1 (b)).
%\end{remark}

%\begin{remark}\indent
%%\begin{enumerate}
% When one  
%%If we solve this circular Radon transform, we get a regular Radon transfom.%This is the motivate of this work.
%
%%\item The distance $u$ between a line detector and origin may depend on $\boldsymbol\theta$ but must be bigger than $R$.  %$[0,2\pi)=\{\boldsymbol\theta:L(\boldsymbol\theta,u,p) \mbox{ for all }u>R \}$. 
%%For example, let $u$ defined by $2R\cos\boldsymbol\theta$ if $-\pi/3\leq \boldsymbol\theta\leq \pi/3$, $2\pi/3\leq\boldsymbol\theta \leq 4\pi/3$ and $2R\sin\boldsymbol\theta$ otherwise. Then a line $L(u,\boldsymbol\theta,p)$ for each $\boldsymbol\theta$ passes through $(\pm 2R,0,p)$ or $(0,\pm2R,p)$.
%
%%\item Haltmeier discussed four inversion formulas for $R_C$ when $R=R$, in~\cite{haltmeier11}.%(all of which come from one formula).
%%Our transform is a little more general than Haltmeier's model.
%%\end{enumerate}
%\end{remark}

By definition, we have
$$
R_Cf(\boldsymbol\theta,p,r)=\displaystyle\frac{1}{2\pi}\intR\int\limits^{\pi}_{-\pi}f(t\boldsymbol\theta^\perp+(R-r\cos\psi)\boldsymbol\theta, p+r\sin\psi)d\psi dt,
$$
{\color{black} where $p$ and $r$ are the height and radius of the cylinder of integration and $\boldsymbol\theta$ is the direction from the $z$-axis to the central axis of the cylinder.
Also, $t$ is a parameter along the central axis of the cylinder, $\boldsymbol\theta^\perp$ in the $xy$-plane is any unit vector perpendicular to $\boldsymbol\theta$,} and $\psi$ is the polar angle of the circle that is the intersection of the plane $\{(t\boldsymbol\theta,z):t\in\RR,z\in\RR\}$ and the cylinder (see Figure 1 (b)).

%%%%%%%%%%%%%%%%%%%%%%%%%%%%%%%%%%%%%%%%%%%%
\subsection{Inversion formulas}\label{recon}

%%%%%%%%%%%%%%%%%%%%%%%%%%%%%%%%%%%%%%%%
We have two integrals in the definition formula of $R_Cf$. 
For fixed $\boldsymbol\theta$, the inner integral is a circular Radon transform with centers on the line $\{(R\boldsymbol\theta,z):z\in\RR^{ }\}$ (see Figure 1 (b)). 
Also, the outer integral can be {\color{black}thought} of as the $2$-dimensional regular Radon transform for a fixed $z$ variable~\cite{haltmeier11}.
We start by applying the inversion of the circular Radon transform for fixed $\boldsymbol\theta$.
% then we get the $2$-dimension regular Radon transform.

To obtain inversion formulas, we define the operator 
 $R^*_C$ for $g\in C^\infty_c(S^{1}\times\RR^{ }\times[0,\infty))$ by
$$
R^*_Cg(\boldsymbol\theta,\zeta,\rho)=\displaystyle\intR g(\boldsymbol\theta,p,\sqrt{(\zeta-p)^2+\rho^2})dp,
$$
for $z\in\RR^{ }$ and $\rho\in\RR$.

We have an analogue of the Fourier slice theorem.
\begin{theorem}\label{thm:inversioncylinder3d}
Let $f\in C^\infty_c(B^{2}_R\times \RR)$. %, where $B^k_R$ is the ball in $\RR^k$ with radius $R$. 
If $g=R_Cf$, then we have for $(\boldsymbol\theta,\sigma,\xi)\in S^{1}\times\RR\times \RR^{}$,
\begin{equation}\label{eq:fourierslice}
\hat{f}(\sigma\boldsymbol\theta,\xi)=\pi^{-1}\widehat{R^*_Cg}(\boldsymbol\theta,\xi,\sigma)e^{-iR\sigma}|\sigma|,
\end{equation}
{\color{black}where $\hat f$ is the 3-dimensional Fourier transform of $f$, i.e.,
$$
\hat{f}(\xi_1,\xi_2,\xi_3)=\intL_{\RR^3}f(x,y,z)e^{-i(x,y,z)\cdot (\xi_1,\xi_2,\xi_3)} dxdydz,
$$
and $\widehat{R^*_Cg}$ is the $2$-dimensional Fourier transform of $R^*_Cg$ with respect to $(\zeta,\rho)$, i.e.,
$$
\widehat{R^*_Cg}(\boldsymbol\theta,\xi,\sigma)=\intR\intR R^*_Cg(\boldsymbol\theta,\zeta,\rho)e^{-i(\zeta,\rho)\cdot(\xi,\sigma)}dpd\rho.
$$}
%where the operator $\partial _\rho$ denotes the derivative in the last variable $\rho$ of $R^*_Cg(\boldsymbol\theta,z,\rho)$.
\end{theorem}
\begin{remark}\label{rmk:fourierslicecylinder}
We remind the readers of the Fourier slice theorems for the circular and regular Radon transforms.

When $\mathcal Rf(\boldsymbol\theta,s)=\int_{\boldsymbol\theta\cdot (x,y)=s}f(x,y)dxdy$ for $(\boldsymbol\theta,s)\in S^1\times\RR$ is the regular Radon transform, we have $\widehat{\mathcal Rf}(\boldsymbol\theta,\sigma)=\hat f(\sigma\boldsymbol\theta)$.
Also, when $Mf(u,r)=\int_{S^1}f((u,0)+r\boldsymbol\alpha)dS(\boldsymbol\alpha)$ for $(u,r)\in\RR\times[0,\infty)$ is the circular Radon transform, we have $\hat f(\xi_1,\xi_2)=\widehat{M^*Mf}(\xi_1,\xi_2)|\xi_2|$, where $dS$ is the standard measure on the unit circle and $M^*g(x,y)=\int_\RR g(u,\sqrt{(u-x)^2+y^2})du$ for a function $g$ on $\RR\times[0,\infty)$.
{\color{black}(For the proof, see~\cite{nattererw01,nilsson97}.)}
Equation~\eqref{eq:fourierslice} can be thought of as the combination of two Fourier slice theorems for the circular and regular Radon transforms.
\end{remark}
\begin{proof}
Taking the Fourier transform of $R_Cf$ with respect to $p$ yields 
$$
\widehat{R_Cf}(\boldsymbol\theta,\xi,r)=\frac{1}{2\pi }\intR\int\limits^1_{-1}\hat{f}(t\boldsymbol\theta^\perp+( R -r\sqrt{1-s^2})\boldsymbol\theta, \xi)e^{irs\xi}\frac{ds}{\sqrt{1-s^2}} dt,
$$
where $\hat{f}$ and $\widehat{R_Cf}$ are the 1-dimensional Fourier transforms of $f$ and $R_Cf$ with respect to $z$ and $p$, respectively.
Taking the Hankel transform of order zero of $\widehat{R_Cf}$ with respect to $r$, we have %for $\eta$%$|\xi|^2=\xi^2_1+\xi^2_2$,
\begin{equation}\label{eq:hankelrcfcylinder}
\begin{array}{ll}
H_0\widehat{R_Cf}(\boldsymbol\theta,\xi,\eta)&=\displaystyle\frac{1}{2\pi }\int\limits\half\intR\int\limits^1_{-1}\hat{f}(t\boldsymbol\theta^\perp+( R -r\sqrt{1-s^2})\boldsymbol\theta, \xi)e^{irs\xi}\frac{ds dt}{\sqrt{1-s^2}}\;J_0(r\eta)rdr\\
   &=\displaystyle\frac{1}{2\pi }\int\limits\half\intR\int\limits^1_{-1}\hat{f}(t\boldsymbol\theta^\perp+( R -r\sqrt{1-s^2})\boldsymbol\theta, \xi)\cos(rs\xi)\frac{rJ_0(r\eta)dsdtdr}{\sqrt{1-s^2}} \\
      &=\displaystyle\frac{1}{2\pi }\intR\int\limits\half\int\limits\half\hat{f}(t\boldsymbol\theta^\perp+( R -b)\boldsymbol\theta, \xi)\cos(\rho\xi) J_0(\eta\sqrt{\rho^2+b^2})d\rho dbdt,
  \end{array}
  \end{equation}
  where in the last line, we changed variables $(r,s)\rightarrow (\rho,b)$ where $r=\sqrt{\rho^2+b^2}$ and $s=\rho/\sqrt{\rho^2+b^2}$.
We will use the following identity: for $a,b>0$%for $0<\xi_1<\eta$,
\begin{equation}\label{eq:batemann2}
\displaystyle\int\limits\half J_0(a\sqrt{\rho^2+b^2})\cos(\rho\xi)d\rho=\left\{\begin{array}{ll}\dfrac{1}{\sqrt{a^2-\xi^2}}\cos(b\sqrt{a^2-\xi^2}) &\mbox{ if } 0<\xi<a,\\
0&\mbox{ otherwise}\end{array}\right.
\end{equation}
\cite[p.55 (35) vol.1]{batemann}.
Applying this identity \eqref{eq:batemann2} to equation~\eqref{eq:hankelrcfcylinder}, we get 
\begin{equation*}%\label{eq:relationhankelandfouriercylinder}
H_0\widehat{R_Cf}(\boldsymbol\theta,\xi,\eta)=\left\{\begin{array}{ll}\displaystyle\frac{1}{2\pi }\intR\int\limits\half\hat{f}(t\boldsymbol\theta^\perp+( R -b)\boldsymbol\theta, \xi)\dfrac{\cos(b\sqrt{\eta^2-\xi^2})}{\sqrt{\eta^2-\xi^2}}dbdt\;&\mbox{ if } 0<\xi<\eta,\\
0&\mbox{ otherwise.}\end{array}\right.
 \end{equation*}
Substituting $\eta=\sqrt{\xi^2+\sigma^2}$ yields
\begin{equation*}
H_0\widehat{R_Cf}(\boldsymbol\theta,\xi,|(\xi,\sigma)|)=\displaystyle\frac{1}{2\pi }\intR\int\limits\half\hat{f}(t\boldsymbol\theta^\perp+( R -b)\boldsymbol\theta, \xi)\dfrac{\cos(b\sigma)}{\sigma}dbdt.
 \end{equation*}
The inner integral in the right hand side is the Fourier cosine transform with respect to $b$, so taking the Fourier cosine transform of the above formula, we get 
\begin{equation}\label{relationhankelandfouriercylinder}
\displaystyle\intR \hat{f}(t\boldsymbol\theta^\perp+( R -s)\boldsymbol\theta, \xi)dt= 4\int\limits\half H_0\widehat{R_Cf}(\boldsymbol\theta,\xi,|(\xi,\sigma)|)\cos(s\sigma)\sigma d\sigma,
\end{equation}
where $\hat{f}$ is the 1-dimensional Fourier transform of $f$ with respect to the last variable $z$.
For a fixed $\xi$, one recognizes the Radon transform in the left side. 
We, thus, can apply the Fourier slice theorem for the regular Radon transform. 

Before doing that, we change the right side of equation~\eqref{relationhankelandfouriercylinder} into a term containing the backprojection operator $R_C ^*$.
Taking the Fourier transform of $R^*_Cg$ on $S^1\times\RR^2$ with respect to the last two variables $(\zeta,\rho)$ yields 
\begin{equation}\label{eq:relationhankelandbackcylinder}
\begin{array}{ll}
   \widehat{R^*_Cg}(\boldsymbol\theta, \xi,\sigma)&=\displaystyle\intR\intR e^{-i(\zeta,\rho)\cdot(\xi,\sigma)} R^*_Cg(\boldsymbol\theta,\zeta,\rho)d\zeta d\rho\\
   &=\displaystyle\intR\intR e^{-i(\zeta,\rho)\cdot (\xi,\sigma)} \intR g(\boldsymbol\theta,p,\sqrt{(\zeta-p)^2+\rho^2} )dpd\zeta d\rho\\
&=\displaystyle\intR e^{-i\xi p}\intR\intR e^{-i(\zeta-p,\rho)\cdot(\xi,\sigma)} g(\boldsymbol\theta,p,\sqrt{(\zeta-p)^2+\rho^2} )d\zeta d\rho dp\\
&=\displaystyle\intR e^{-i\xi p} \intR\intR e^{-i(\zeta,\rho)\cdot (\xi,\sigma)} g(\boldsymbol\theta,p,|(\zeta,\rho)| )d\zeta d\rho dp\\
&\displaystyle=2\pi\displaystyle\intR e^{-i\xi p}H_0g(\boldsymbol\theta,p,|(\xi,\sigma)| )dp\\
&=2\pi H_0\hat{g}(\boldsymbol\theta,\xi,|(\xi,\sigma)| ).
  \end{array}
  \end{equation}
%  where $\widehat{R^*_Cg}$ is the 2-dimensional Fourier transform of $g$ in $(z,\rho)$.
Combining equation~\eqref{eq:relationhankelandbackcylinder} with equation~\eqref{relationhankelandfouriercylinder}, we have for $g=R_Cf$,
\begin{equation}\label{relationfourierandbackcylinder}
\begin{array}{ll}
\displaystyle \intR \hat{f}(t\boldsymbol\theta^\perp+s\boldsymbol\theta, \xi)dt &\displaystyle=\frac{2}{\pi} \int\limits\half \widehat{R^*_Cg}(\boldsymbol\theta, \xi,\sigma) \cos((R-s)\sigma)\sigma d\sigma\\
&\displaystyle=\frac{2}{\pi} \int\limits\half \widehat{R^*_Cg}(\boldsymbol\theta, \xi,\sigma) \cos((s-R)\sigma)\sigma d\sigma\\
&\displaystyle=\frac{1}{\pi} \intR \widehat{R^*_Cg}(\boldsymbol\theta, \xi,\sigma) e^{i(s-R)\sigma}|\sigma| d\sigma,
 \end{array}
\end{equation}
since $\widehat{R_C^*g}$ is even in $\sigma$ by the evenness in $\rho$ of $R^*_Cg$.
Taking the Fourier transform of equation~\eqref{relationfourierandbackcylinder} with respect to $s$ completes the proof.
\end{proof}

\begin{theorem}\label{thm:3dcylinder}
Let $f\in C^\infty_c(B^2_R\times\RR^{})$. If $g=R_Cf$, then we have 
\begin{equation}\label{eq:inversioncyliner}
f(x,y,z)=-\frac{1}{4\pi^2}\intL_{S^{1}}\left.\frac{\partial^2}{\partial\rho^2}R^*_Cg(\boldsymbol\theta,z,\rho)\right|_{\rho=(x,y)\cdot\boldsymbol\theta-R}dS(\boldsymbol\theta),
\end{equation}
where $dS$ is the standard measure on the unit circle $S^1$.
%where we use the Riesz potential $\widehat{I^{-2}_2h}(\boldsymbol\theta,\xi,\sigma)=|\sigma|^2\hat h(\boldsymbol\theta,\xi,\sigma)$ for a function $h(\boldsymbol\theta,z,\rho)$ on $S^1\times\RR^2$ with its $2$-dimensional Fourier transform $\hat h(\boldsymbol\theta,\xi,\sigma)$ with respect to real variables $(z,\rho)$.
\end{theorem}

\begin{proof}
Using Theorem~\ref{thm:inversioncylinder3d}, we have 
$$
\begin{array}{ll}
f(x,y,z)&\displaystyle=\frac{1}{(2\pi)^3}\intL\half\intL_{S^{1}}\intL_{\RR^{}}\hat{f}(\sigma\boldsymbol\theta,\xi)|\sigma|^{}e^{i(\sigma (x,y)\cdot \boldsymbol\theta+z \xi)}d\sigma dS(\boldsymbol\theta) d\xi\\
&\displaystyle=\frac{1}{(2\pi)^{3}\pi} \intL\half\intL_{S^{1}}\intL_{\RR^{}}\widehat{R^*_Cg}(\boldsymbol\theta,\xi,\sigma)e^{-iR\sigma}|\sigma|^2e^{i(\sigma (x,y)\cdot \boldsymbol\theta+z \xi)}d\sigma dS(\boldsymbol\theta) d\xi\\
&=\displaystyle\frac{1}{(2\pi)^{4}}\intR\intL_{S^{1}}\intL_{\RR^{}}\widehat{R^*_Cg}(\boldsymbol\theta,\xi,\sigma)e^{-iR\sigma}|\sigma|^2e^{i(\sigma (x,y)\cdot \boldsymbol\theta+z \xi)}d\xi dS(\boldsymbol\theta) d\sigma.
\end{array}
$$
\end{proof}
%Equation~\eqref{eq:inversioncyliner} can be written as
%\begin{equation}\label{eq:inversioncyliner1}
%f(x,y,z)=-\frac{1}{4\pi^2}\intL_{S^{1}}\left.\partial^2_\rho R^*_Cg(\boldsymbol\theta,z,\rho)\right|_{\rho=(x,y)\cdot\boldsymbol\theta-R}dS(\boldsymbol\theta).
%\end{equation}
\begin{remark}\label{rmk:haltmeier}
{\color{black}Inversion formula~\eqref{eq:inversioncyliner} is the same as Theorem 1.2 of~\cite{haltmeier11} if 
$$
\partial_\rho\intR R_Cf(\boldsymbol\theta,p,\sqrt{(z-p)^2+\rho^2})dp=\intR \partial_\rho R_Cf(\boldsymbol\theta,p,\sqrt{(z-p)^2+\rho^2})dp. 
$$
This statement follows from the chain rule, i.e.,
$$
\partial_\rho R_Cf(\boldsymbol\theta,p,\sqrt{(z-p)^2+\rho^2})=\frac{\rho}{\sqrt{(z-p)^2+\rho^2}}\partial_rR_Cf(\boldsymbol\theta,p,\sqrt{(z-p)^2+\rho^2})).
$$}
There M. Haltmeier obtained it by combining two inversion formulas for the circular Radon transform and the 2-dimensional Radon transform. 
On the other hand, we obtain it through an analogue of the Fourier slice theorem.
\end{remark}
 
%%%%%%%%%%%%%%%%%%%%%%%%%%%%%%%%%%%%%%%%%%%%%%
%%%%%%%%%%%%%%%%%%%%%%%%%%%%%%%%%%%%%%%%%%%%%%%%%%

%The regular Radon transform can be obtained from the cylindrical Radon transform $R_C$ in another manner. 

Equation~\eqref{relationfourierandbackcylinder} hints that it is natural to try to use another inversion of the Radon transform, namely the one using circular harmonics.
Let $f(t,\varphi,z)$ be the image function in cylindrical coordinates, where $t=|(x,y)|\in[0,\infty)$ and $(\cos\varphi,\sin\varphi)=(x,y)/|(x,y)|\in [0,2\pi)$.
Then the Fourier series of $f(t,\varphi,z)$ and $g(\boldsymbol\theta,p,r):=R_Cf(\boldsymbol\theta,p,r)$ with respect to their angular variables $\varphi$ and $\boldsymbol\theta$ can be written as follows:
$$
f(t,\varphi,z)=\displaystyle \sum^\infty_{l=-\infty}f_l(t,z)\,e^{il\varphi} \quad\mbox{and}\quad g(\boldsymbol\theta,p,r)=\displaystyle \sum^\infty_{l=-\infty}g_l(p,r)\,e^{il\vartheta},
$$
where $\boldsymbol\theta=(\cos\vartheta,\sin\vartheta)\in S^1$ and the Fourier coefficients are given by
$$
f_l(t,z)=\displaystyle\frac{1}{2\pi}\intL^{2\pi}_0f(t,\varphi,z)\,e^{-il\varphi}d\varphi \quad\mbox{and}\quad g_l(p,r)=\displaystyle\frac{1}{2\pi}\intL^{2\pi}_0g(\boldsymbol\theta,p,r)\,e^{-il\vartheta}d\vartheta.
$$
 %where $\boldsymbol\theta=(\cos\vartheta,\sin\vartheta)$. 
Consider the $l$-th Fourier coefficient of the right hand side of formula~\eqref{relationfourierandbackcylinder}.
Then we have
\begin{equation}\label{eq:relationfourierandback}
%\begin{array}{ll}
\intL_{S^1}\intR\widehat{R^*_Cg}(\boldsymbol\theta,\xi,\sigma)e^{i(R-s)\sigma}|\sigma| e^{-il\vartheta} d\sigma dS(\boldsymbol\theta)=\intR\widehat{R^*_Cg_l}(\xi,\sigma)e^{i(R-s)\sigma}|\sigma| d\sigma,
\end{equation}
where $\widehat{R^*_Cg_l}$ is the 2-dimensional Fourier transform of $R^*_Cg_l$ with respect to $(\zeta,\rho)$ and
$$
R^*_Cg_l(\zeta,\rho)=\intR g_l(p,\sqrt{(\zeta-p)^2+\rho^2})dp.
$$
According to~\cite{natterer01}, when $g=\mathcal Rf$ for the regular 2-dimensional Radon transform $\mathcal R$ and $g_l$ and $f_l$ are the $l$-th Fourier coefficients of $g$ and $f$, we have for $t>0$,
\begin{equation}\label{eq:sphericalharmonics}
f_l(t)=-\pi^{-1} \intL^\infty_t\cosh\left(l\operatorname{arccosh}\frac{s}{t}\right)\frac{\partial}{\partial_s}g_l(s)\frac{ds}{\sqrt{s^2-t^2}}.
\end{equation}
%where $P_n$ is the Legendre polynomials.
Hence, we get a different type of an inversion formula.
\begin{theorem}\label{thm:cormack}
Let $f\in C^\infty_c(B^2_R\times\RR)$. Then we have for $t>0$
$$
f_l(t,z)=\frac{2}{\pi}\intL^\infty_t \cosh\left(l\operatorname{arccosh}\frac{s}{t}\right)H_\rho \frac{\partial^2}{\partial^2_\rho} R^*_Cg_l(z,s-R)  \frac{ds}{\sqrt{s^2-t^2}},
$$
where $H_\rho h(\zeta,\rho)$ is the Hilbert transform of $h(\zeta,\rho)$ on $\rho$.
\end{theorem}
\begin{proof}
Combining the three equations \eqref{relationfourierandbackcylinder},~\eqref{eq:relationfourierandback}, and ~\eqref{eq:sphericalharmonics} gives 
\begin{equation*}
\begin{split}
\widehat{f_l}(t,\xi)%&\displaystyle=-\frac{2i}{\pi^2} \intL^\infty_t \cosh\left(l\operatorname{arccosh}\frac{s}{t}\right)\intL\half\widehat{R^*_Cg_l}(\xi)\cos((u-s)\sigma)|\sigma|^2 d\sigma \frac{ds}{\sqrt{s^2-t^2}}\\
&\displaystyle=-\frac{1}{\pi^2} \intL^\infty_t \cosh\left(l\operatorname{arccosh}\frac{s}{t}\right)\intR i\operatorname{sgn}(\sigma)\widehat{R^*_Cg_l}(\xi,\sigma)e^{i(s-R)\sigma}\sigma^2 d\sigma \frac{ds}{\sqrt{s^2-t^2}}\\
&\displaystyle=\frac{1}{\pi^2} \intL^\infty_t \cosh\left(l\operatorname{arccosh}\frac{s}{t}\right)\intR (H_\rho\partial^2_\rho R^*_Cg_l)\widehat{\;  }\;(\xi,\sigma)e^{i(s-R)\sigma} d\sigma \frac{ds}{\sqrt{s^2-t^2}},
\end{split}
\end{equation*}
%\begin{equation*}
%\begin{split}
%\widehat{f_l}(t,\xi)%&\displaystyle=-\frac{2i}{\pi^2} \intL^\infty_t \cosh\left(l\operatorname{arccosh}\frac{s}{t}\right)\intL\half\widehat{R^*_Cg_l}(\xi)\cos((u-s)\sigma)|\sigma|^2 d\sigma \frac{ds}{\sqrt{s^2-t^2}}\\
%&\displaystyle=-\frac{1}{\pi^2} \intL^\infty_t \cosh\left(l\operatorname{arccosh}\frac{s}{t}\right)\intR i\operatorname{sgn}(\sigma)\widehat{R^*_Cg_l}(\xi,\sigma)e^{i(s-R)\sigma}\sigma^2 d\sigma \frac{ds}{\sqrt{s^2-t^2}}\\
%&\displaystyle=\frac{1}{\pi^2} \intL^\infty_t \cosh\left(l\operatorname{arccosh}\frac{s}{t}\right)\intR (H_\rho\partial^2_\rho R^*_Cg_l)\widehat{\;  }\;(\xi,\sigma)e^{i(s-R)\sigma} d\sigma \frac{ds}{\sqrt{s^2-t^2}},
%\end{split}
%\end{equation*}
where $\widehat{f_l}$ is the 1-dimensional Fourier transform of $f_l$ with respect to $z$, and in the last line, we used the identity $\widehat {H_\rho h}(\sigma)=-i\operatorname{sgn} (\sigma)\hat h(\sigma)$.
\end{proof}

%%%%%%%%%%%%%%%%%%%%%%%%%%%%
%In this section, 
The regular Radon transform can be obtained from the cylindrical Radon transform. 

\begin{theorem}\label{thm:reddingandradon}
  Let $f\in C^\infty_c(B^2_R\times\RR)$.
 % If $$, 
 Then we have %$f(x,y,z)$ can be found as
\begin{equation*}\label{eq:reddingcrt}
\displaystyle\intR f(t\boldsymbol\theta^\perp+( R -s)\boldsymbol\theta,z)dt=\displaystyle\frac{2}{\pi}\intRR\int\limits\half s rR_Cf(\boldsymbol\theta,-\eta,r) e^{-ir^2\xi}e^{-i(2z\eta+(z^2+s^2)+\eta^2)\xi}\xi drd\eta d\xi.
  \end{equation*}
%$$
%f(x,y,z)=-\displaystyle\frac{i}{\pi^2}\int\limits^{\pi}_{-\pi}\intRR\int\limits\half rR_Cf(\boldsymbol\theta,-\eta,r)  h_u(\xi,\boldsymbol\theta\cdot(x,y))e^{-i(2z\eta+z^2+r^2-\eta^2)\xi}\xi drd\eta d\xi dS(\boldsymbol\theta),
%$$
%where 
%$$
%h_u(\xi,t)=H_s \left(\frac{\partial}{\partial s}(u-s)e^{-i(u-s)^2\xi}\right)(t).
%$$
%Here $H_s$ is Hilbert transform with respect to $s$.
\end{theorem}

We notice that the expression in the left hand side is the standard $2$-dimensional Radon transform for a fixed $z$ variable. Hence, applying different Radon transform inversions, one gets different inversions of the cylindrical Radon transform $R_C$.
We will follow the idea suggested in~\cite{reddingn01} to prove this Theorem~\ref{thm:reddingandradon}.

\begin{proof}
Let $G$ be defined by 
$$
G(\boldsymbol\theta,p,\xi):=\displaystyle\int\limits\half rR_Cf(\boldsymbol\theta,p,r) e^{-ir^2\xi}dr.
$$
Then we have
$$
\begin{array}{ll}
   G(\boldsymbol\theta,p,\xi)&=\displaystyle\frac{1}{2\pi}\int\limits\half\intR\int\limits^\pi_{-\pi}rf(t\boldsymbol\theta^\perp+( R -r\cos\psi)\boldsymbol\theta,p+r\sin\psi)e^{-ir^2\xi}d\psi dtdr\\
&=\displaystyle\frac{1}{2\pi}\int\limits_{\RR^3} f(t\boldsymbol\theta^\perp+( R -y)\boldsymbol\theta,p+z) e^{-i(y^2+z^2)\xi}dydzdt\\
&=\displaystyle\frac{1}{2\pi}\int\limits_{\RR^3} f(t\boldsymbol\theta^\perp+( R -y)\boldsymbol\theta,z) e^{-i(y^2+(z-p)^2)\xi}dydzdt\\
&=\displaystyle \frac{e^{-ip^2\xi}}{2\pi}\int\limits_{\RR^3} f(t\boldsymbol\theta^\perp+( R -y)\boldsymbol\theta,z) e^{-i(y^2+z^2)\xi}e^{2ipz\xi}dydzdt,
  \end{array}
  $$
where in the second line, we switched from the polar coordinates $(r,\psi)$ to Cartesian coordinates $(y,z)$.
Making the change of variables $r=y^2+z^2$ yields
\begin{equation}\label{eq:gthetapxi}
\begin{array}l
   G(\boldsymbol\theta,p,\xi)=\displaystyle \frac{e^{-ip^2\xi}}{2\pi}\int\limits_{\RR^2}\intL\half f(t\boldsymbol\theta^\perp+( R -\sqrt{r-z^2})\boldsymbol\theta,z) \frac{e^{-ir\xi}e^{2ipz\xi}}{2\sqrt{r-z^2}}drdzdt.
  \end{array}
\end{equation}
(We do not need to account for $f(t\boldsymbol\theta^\perp+( R +\sqrt{r-z^2})\boldsymbol\theta,z)$ because $f$ is compactly {\color{black}supported} on $B^2\times\RR$.)
Let us define the function 
\begin{equation}\label{eq:ktheta}
k_{\boldsymbol\theta}(t,z,r):=\left\{\begin{array}{ll}f(t\boldsymbol\theta^\perp+( R -\sqrt{r-z^2})\boldsymbol\theta,z)/\sqrt{r-z^2} \qquad &\mbox{if } 0<z^2<r,\\
                        0\qquad &\mbox{otherwise}.
                       \end{array}\right.
\end{equation}
Applying equation~\eqref{eq:ktheta} into \eqref{eq:gthetapxi} gives
$$
%\begin{array}{ll}
   G(\boldsymbol\theta,p,\xi)=\displaystyle\frac{e^{-ip^2\xi}}{4\pi}\int\limits_{\RR^3} k_{\boldsymbol\theta}(t,z,r) e^{-ir\xi}e^{2ipz\xi}drdzdt=\displaystyle\frac{e^{-ip^2\xi}}{4\pi}\intR\widehat{k_{\boldsymbol\theta}}(t,-2p\xi,\xi)dt,
 % \end{array}
  $$
  where $\widehat{k_{\boldsymbol\theta}}$ is the 2-dimensional Fourier transform of $k_{\boldsymbol\theta}$ with respect to the last two variables $(z,r)$.
Also, we have
\begin{equation*}\label{reddingandradon}
\begin{array}{ll}
\displaystyle\intR f(t\boldsymbol\theta^\perp+( R -s)\boldsymbol\theta,z)dt&=\displaystyle \intR sk_{\boldsymbol\theta}(t,z,z^2+s^2)dt\\
&=\displaystyle\frac{1}{4\pi^2}\intR\intR \intR s\widehat{k_{\boldsymbol\theta}}(t,\eta,\xi)e^{-i(z\eta+(z^2+s^2)\xi)}dt d\eta d\xi\\
&=\displaystyle\frac{1}{\pi}\intR\intR se^{i\frac{\eta^2}{4\xi}}G\left(\boldsymbol\theta,-\frac{\eta}{2\xi},\xi\right)e^{-i(z\eta+(z^2+s^2)\xi)}d\eta d\xi\\
&=\displaystyle\frac{2}{\pi}\intR\intR sG(\boldsymbol\theta,-\eta,\xi)e^{-i(2z\eta+(z^2+s^2)+\eta^2)\xi}\xi d\eta d\xi,
  \end{array}
  \end{equation*}
  where in the last line, we changed variables $\eta\rightarrow 2\xi\eta$.
%Since the left side is Radon transform for a fixed $z$ axis, applying to a standard inversion formula of regular Radon transform yields
%$$
%\begin{array}{ll}
%f(x,y,z)=-\displaystyle\frac{i}{\pi^2}\int\limits^{\pi}_{-\pi}\intRR G(\boldsymbol\theta,-\eta,\xi)  e^{-i(2z\eta+z^2-\eta^2)\xi} h_u(\xi,\boldsymbol\theta\cdot(x,y)) \xi d\eta d\xi dS(\boldsymbol\theta).
%  \end{array}
%  $$
\end{proof}

%Similar to theorem~\ref{cormack}, applying formula~\eqref{eq:sphericalharmonics} to~\eqref{reddingandradon} gives another inversion.
%\begin{theorem}
%Let $f\in C^\infty_c(\Omega_R)$. If $u\geq R$, then for $\rho>0$, $f_n(\rho,z)$ is obtained by
%$$
%-\frac{2}{\pi^2}\intL^\infty_\rho \intRR \intL\half P_n\left(\frac{s}{\rho}\right)g_n(-\eta,r)(6(u-s)i\xi^2-4(u-s)^2\xi^3)e^{-i\xi(r^2+2z\eta+z^2+(u-s)^2-\eta^2)} drd\eta d\xi ds.
%$$
%\end{theorem}

%%%%%%%%%%%%%%%%%%%%%%%%%%%%%%%%%%%%%%%%%
\subsection{Support theorem}\label{sec:uniquess}
%%%%%%%%%%%%%%%%%%%%%%%%%%%%%%%%%%%%%
From the inversion formulas in subsection~\ref{recon}, we know that $f\in C^\infty_c(B^2_R\times\RR)$ is uniquely recovered from $R_Cf$. 
In many practical situations, we know the data (in this case, $R_Cf$) only on a subset of their domain. 
Hence, it is important that these partial data still determine $f$ uniquely.

By a support theorem, we mean a statement that claims that if integrals of $f$ over all surfaces not intersecting a set $A$ are equal to zero, then $f$ is equal to zero outside $A$.
%This statement cannot hold for arbitrary $A$, but under some geometry restrictions of convexity type.
%In this subsection, we show that a certain amount partial data of $R_Cf$ still determines $f$ uniquely.% hole or support theorem of cylindrical Radon transform. 
\begin{lemma}\label{lem:unique}
Let $p_0\in\RR$, $\epsilon>0$, $B>0$, and $\boldsymbol\theta\in S^1$ be given.
Let $f\in C^\infty(B^2_R\times\RR)$ and suppose that $g=R_Cf$ is equal to zero on the open set $U_{B,\epsilon}=\{(p,r):|p-p_0|<\epsilon,0\leq r<B\}$. Then $\mathcal R_{\boldsymbol\theta} f(p,s)$ is equal to zero on the open set $V_B=\{(z,s):|z-p_0|^2+( R -s)^2<B^2 \}$ where 
$$
\mathcal R_{\boldsymbol\theta} f(z,s)=\displaystyle\intR f(t\boldsymbol\theta^\perp+s\boldsymbol\theta,z)dt.
$$
\end{lemma}
We will follow the idea suggested in~\cite{andersson88} to prove this lemma.
\begin{proof}
Without loss of generality, we may assume $p_0=0$.
Let $G(\boldsymbol\theta,p,r)$ be defined by
$$
G(\boldsymbol\theta,p,r)=\intL^r_0g(\boldsymbol\theta,p,s)sds=\frac1{2\pi}\intL^r_0\intR\intL^{\pi}_{-\pi}f(t\boldsymbol\theta^\perp+(R-s\cos\psi)\boldsymbol\theta,p+s\sin\psi)sd\psi dtds.
$$
Changing the variables $s(\cos\psi,\sin\psi)\to\boldsymbol\eta=(\eta_1,\eta_2)\in\RR^2$, we have
$$
G(\boldsymbol\theta,p,r)=\frac{1}{2\pi}\intL_{|\boldsymbol\eta|\leq r}\intR f(t\boldsymbol\theta^\perp+( R -\eta_1)\boldsymbol\theta,p+\eta_2) dt d\boldsymbol\eta.
$$ 
%where $\eta=(\eta_1,\eta_2)\in \RR^2$.
Differentiating $G$ with respect to $p$ yields
$$
\begin{array}{ll}
\dfrac{\partial}{\partial p} G(\boldsymbol\theta,p,r)&\displaystyle=\frac{1}{2\pi}\intL_{|\boldsymbol\eta|\leq r}\intR \frac{\partial}{\partial p}f(t\boldsymbol\theta^\perp+( R -\eta_1)\boldsymbol\theta,p+\eta_2) dt d\boldsymbol\eta\\
&\displaystyle=\frac{1}{2\pi}\intL_{|\boldsymbol\eta|\leq r}\intR \frac{\partial}{\partial \eta_2}f(t\boldsymbol\theta^\perp+( R -\eta_1)\boldsymbol\theta,p+\eta_2) dt d\boldsymbol\eta\\
&\displaystyle=\frac{1}{2\pi r}\intL_{|\boldsymbol\eta|= r}\intR f(t\boldsymbol\theta^\perp+( R -\eta_1)\boldsymbol\theta,p+\eta_2)\eta_2 dt d\boldsymbol\eta,
\end{array}
$$
where in the last line, we used the divergence theorem.
Now we have
$$
\begin{array}{ll}
R_C(zf)(\boldsymbol\theta,p,r)&\displaystyle=\frac{1}{2\pi r}\intR\intL_{|\boldsymbol\eta|=r}(p+\eta_2)f(t\boldsymbol\theta^\perp+( R -\eta_1)\boldsymbol\theta,p+\eta_2) d\boldsymbol\eta dt\\
&=pg(\boldsymbol\theta,p,r)+\dfrac{1}{2\pi}\dfrac{\partial}{\partial p} G(\boldsymbol\theta,p,r)=\displaystyle  pg(\boldsymbol\theta,p,r)+\frac{\partial}{\partial p}\intL^r_0g(\boldsymbol\theta,p,s)sds.
\end{array}
$$
Let the linear operator $D$ be defined by $Dg(\boldsymbol\theta,p,r):=pg(\boldsymbol\theta,p,r)+\frac{\partial}{\partial p}\int^r_0g(\boldsymbol\theta,p,s)sds$. Then $R_C(zf)$ is equal to $Dg$. By iteration, we have $R_C(\mathcal{P}(z)f)=\mathcal{P}(D)g$ where $\mathcal{P}$ is any polynomial. If $g=0$ in $U_{B,\epsilon}$, then $\mathcal{P}(D)g=0$ in $U_{B,\epsilon}$. Also, we have for any point $(p,r)\in U_{B,\epsilon}$,
$$
\begin{array}{ll}
R_C(\mathcal{P}(z)f)(\boldsymbol\theta,p,r)&\displaystyle=\frac{1}{2\pi r}\intL_{|\boldsymbol\eta|=r}\intR \mathcal{P}(p+\eta_2)f(t\boldsymbol\theta^\perp+( R -\eta_1)\boldsymbol\theta,p+\eta_2)dt d\boldsymbol\eta\\
&\displaystyle=\frac{1}{2\pi}\intL^r_{-r}\intR \mathcal{P}(p+\eta_2)f(t\boldsymbol\theta^\perp+( R -\sqrt{r^2-\eta^2_2})\boldsymbol\theta,p+\eta_2) \frac{dtd\eta_2}{\sqrt{r^2-\eta^2_2}}\\
&=0.
\end{array}
$$
Let $\boldsymbol\theta\in S^1$, $r>0$, and $p\in\RR$ be fixed.
By the Stone-Weierstrass Theorem, we can choose a sequence of polynomials $\mathcal{P}_i$ such that $\mathcal{P}_i(p+\eta_2)$ converges to $\int_{\RR}f(t\boldsymbol\theta^\perp+( R -\sqrt{r^2-\eta^2_2})\boldsymbol\theta,p+\eta_2) dt$ uniformly for $|\eta_2|<r$. It follows that $\mathcal R_{\boldsymbol\theta} f(z,s)=0$ in $V_B$.
\end{proof}
\begin{theorem}\label{thm:support}
Let $p_0\in\RR$ and $B>0$.
Let $f\in C^\infty(B^2_R\times\RR)$ and suppose that $g=R_Cf$ is equal to zero on the open set $U_{B}=\{(\boldsymbol\theta,p_0,r):0\leq\boldsymbol\theta<2\pi,0\leq r<B\}$. Then $f$ is equal to zero on the set $\{(x,y,z):|(x,y)|> R-\sqrt{B^2-(z-p_0)^2} \}$.
\end{theorem}
\begin{proof}
Let $\epsilon>0$ be arbitrary.
Then $g$ vanishes on the open set $U_{B-\epsilon,\epsilon}$ and by Lemma~\ref{lem:unique}, $\mathcal R_{\boldsymbol\theta} f$ vanishes on the open set $V_{B-\epsilon}$.
Let $z\in\RR$ be arbitrary.
Notice that $\mathcal R_{\boldsymbol\theta} f(z,s)$ is equal to zero for $s>R$.
Then by the support theorem of the regular Radon transform~\cite{helgason99radon,natterer01}, $f$ is equal to zero on the set $\{(x,y,z)\in\RR^3:|(x,y)|> R-\sqrt{(B-\epsilon)^2-(z-p_0)^2} \}$.
%Since $\epsilon>0$ be arbitrary, we have our assertion.
\end{proof}

% \begin{remark}
% Louis and Quinto discussed a general version of Theorem~\ref{thm:support} in~\cite{louisq00}.
% \end{remark}
\begin{corollary}
Let $A\subset B^2_R\times\RR$ be a closed set invariant under rotation around $z$-axis and let $f\in C^\infty(B^2_R\times\RR)$.
Suppose that for any point $(x,y,z)\in \RR^3\setminus A$, there are $(p_{(x,y,z)},r_{(x,y,z)})\in\RR\times(0,\infty)$ such that a sphere centered at $(Rx/|(x,y)|,Ry/|(x,y)|,p_{(x,y,z)})$ with radius $r_{(x,y,z)}$  separates the point $(x,y,z)$ and $A$.
If $g=R_Cf$ vanishes on $\{(\boldsymbol\theta,p,r):p=p_{(x,y,z)},0\leq r<r_{(x,y,z)},\mbox{ for any } (x,y,z)\in \RR^3\setminus A\} $,
then $f$ vanishes on $\RR^3\setminus A$.
\end{corollary}

%\begin{theorem}\label{thm:hole}
%Let $f\in C^\infty(B^2_R\times\RR)$ and let $R'<R$. %<u$.
%If $R_Cf(\boldsymbol\theta, p, r)=0$ for all $r< R -R'$, then $f=0$ outside $B^2_{R'}\times\RR$.
%\end{theorem}
% \begin{remark}\label{rmk:hole}
%We can get the same result for general dimensional case.
% \end{remark}
% 
%This immediately follows from the hole theorem of the regular 2-dimensional Radon transform and the following lemma. 
%\begin{theorem}\label{thm:hole2}
%Let %$u>R$ and 
%$p_0\in \RR$.
%If $f\in C^\infty(B^2_R\times\RR)$ and $R_Cf(\boldsymbol\theta,p,r)=0$ for all $p>p_0$ (or $p<p_0$) and $r<R$, then $f$ is equal to zero on $\{(x,y,z)\in\RR^3:z>p_0\}\cup\{(x,y,z)\in\RR^3: z<p_0,(R-(x,y)\cdot\boldsymbol\theta)^2+(z-p_0)^2\leq R^2\}$ (or $\{(x,y,z)\in\RR^3:z<p_0\}\cup\{(x,y,z)\in\RR^3: z>p_0, (R-(x,y)\cdot\boldsymbol\theta)^2+(z-p_0)^2\leq R^2\}$ ).
%%Similarly, if $f\in C^\infty(\Omega_R)$ and $R_Cf(\boldsymbol\theta,p,r)=0$ for all $p<p_0$ and $r<u+R$, then $f$ is equal to zero for $z<p_0$.
%\end{theorem}
%\begin{proof}
%Let $z>p_0$ be fixed.
%By Lemma~\ref{lem:unique}, $\mathcal R_{\boldsymbol\theta} f(p,s)$ is equal to zero for any $\boldsymbol\theta$ and $s<R$.
% \end{proof}

%%%%%%%%%%%%%%%%%%%%%%%%%%%%%%%%%%%%%%%%%%%%
\subsection{A stability estimate}\label{estiandrangecylinder}
%%%%%%%%%%%%%%%%%%%%%%%%%%%%%%%%%%%%%%%%%%%%%%%%
In this subsection, %we assume that $R=R\geq R$. Under this assumption, 
we discuss the stability estimate of the cylindrical Radon transform $R_Cf$.
This estimate implies the situation where small errors in the data $g$ lead to small errors in the reconstructed function $f$.
%show that the cylindrical Radon transform is a continous mapping on sobolve sapce.
For the purpose of using them in later sections in which we consider functions on $\RR^n$, we define our Sobolev space for $\RR^n$. %$n$-dimensional spaces.
For $\gamma\in\RR$, let $\mathcal H^{\gamma}(\RR^n)$ be the regular Sobolev space with the norm $||\cdot||_{\gamma}$, {\color{black}i.e.,
$$
\mathcal H^{\gamma}(\RR^n)=\{f\in \mathcal S'(\RR^n):||f||_\gamma<\infty\}
$$
and
$$
||f||_\gamma^2=\intL_{\RR^n}(1+|\boldsymbol\xi|^2)^\gamma|\hat f(\boldsymbol\xi)|^2d\boldsymbol\xi,
$$
where $\mathcal S'(\RR^n)$ is the space of tempered distributions and $\hat f$ is the $n$-dimensional Fourier transform of $f$. }
Let $L^2_{n-k} (S^{k-1}\times\RR^{n-k}\times[0,\infty))$ be the set of {\color{black}functions} $g$ on $S^{k-1}\times\RR^{n-k}\times[0,\infty) $ with %$||g||<\infty$, 
$$
\displaystyle||g||^2:=
\intL_{S^{k-1}}\intL_{\RR^{n-k}}\intL\half |g(\boldsymbol\theta,\mathbf p,r)|^2r^{n-k} dr d\mathbf p dS(\boldsymbol\theta)<\infty.
$$
Here $dS$ is the standard measure on the unit sphere $S^{k-1}$.
Then $L^2_{n-k}(S^{k-1}\times\RR^{n-k}\times[0,\infty))$ is a Hilbert space.
Also, by the Plancherel formula, we have $||g||=(2\pi)^{\frac{-k+n-1}2}||\tilde g||$, where 
$$
\tilde g(\boldsymbol\theta,\boldsymbol\xi,|\boldsymbol\zeta|):=\intL_{\RR^{n-k}}\intL_{\RR^{n-k+1}}g(\boldsymbol\theta,\mathbf p,|\mathbf w|)e^{-i(\mathbf p,\mathbf w)\cdot(\boldsymbol\xi,\boldsymbol\zeta)}d\mathbf pd\mathbf w.
$$ 
%For a function $g\in L^2_{n-k}(S^{k-1}\times\RR^{n-k}\times[0,\infty))$, let a norm $||g||_\gamma$ be defined by
Let $\mathcal H^\gamma (S^{k-1}\times\RR^{n-k}\times[0,\infty))$ be the set of {\color{black}functions} $g\in L^2_{n-k}(S^{k-1}\times\RR^{n-k}\times[0,\infty)) $ with $||g||_\gamma<\infty$, 
where
$$
%\displaystyle||f||_{(1-n)/2}^2=\intL_{\RR^{n-k}}\intL_{\RR^{k}} (1+|(\iota,\xi)|^2)^{(1-n)/2}|\hat{f}(\iota,\xi)|^2d\iota d\xi \mbox{ and }
\displaystyle||g||^2_\gamma:=\intL_{S^{k-1}}\intL_{\RR^{n-k}}\intL\half |\tilde {g}(\boldsymbol\theta,\boldsymbol\xi,\eta)|^2(1+|\boldsymbol\xi|^2+|\eta|^2)^\gamma \eta^{n-k} d\eta d\boldsymbol\xi dS(\boldsymbol\theta).
$$

\begin{theorem}\label{thm:norm}
For $\gamma \geq0$, we have
$$%\begin{equation*}\label{eq:estimate}
||f||_{\gamma}\leq 4\pi^{-1}\ ||R_Cf||_{\gamma+1},
$$ %\end{equation*}
for $f\in \mathcal H^{\gamma}(\RR^3)$ supported in $B^{2}_R\times\RR^{}$ (i.e., $n=3$ and $k=2$).
\end{theorem}
Theorem~\ref{thm:norm} implies that the cylindrical Radon transform is well-posed in the sense that if $f$ satisfying $Cf=g$ is uniquely determined for any $g\in \mathcal H^\gamma (S^{k-1}\times\RR^{n-k}\times[0,\infty))$, the function $f$ depends continuously on $g$. 

\begin{remark}
As mentioned before, $R_C$ can be though of as the composition of the circular Radon transform and the regular Radon transform.
We know that the regular Radon transform maps $H^\gamma(\RR^2)$ into $H^{\gamma+1/2}(S^1\times\RR)$ and the circular Radon transform maps $H^\gamma(\RR^2)$ into $H^{\gamma+1/2}(\RR\times[0,\infty))$ which is defined by the norm 
$$
\intL_{\RR}\intL\half |\tilde\phi(\xi,\rho)|^2(1+|\xi|^2+\rho^2)^{\gamma+1/2}\rho d\rho d\xi <\infty
$$
in~\cite{andersson88,natterer01}.
Hence, the estimate in Theorem~\ref{thm:norm} looks reasonable.
\end{remark}
\begin{proof}
Let $g=R_Cf$.
Note that from equation~\eqref{eq:relationhankelandbackcylinder}, we have 
\begin{equation}\label{eq:hatandtildecylinder}
\widehat{R^*_Cg}(\boldsymbol\theta,\xi,\sigma)=\intL_{\RR^{}}e^{-i\xi  p}\intR\intL_{\RR^{}}e^{-i(\zeta,\rho)\cdot(\xi,\sigma)}g(\boldsymbol\theta,p,|(\zeta,\rho)|)d\zeta d\rho dp=\tilde {g}(\boldsymbol\theta,\xi,|(\xi,\sigma)|).
\end{equation}
Combining this equation and equation~\eqref{eq:fourierslice}, we have 
$$
\hat f(\sigma\boldsymbol\theta,\xi)=4\pi^{-1}\tilde {g}(\boldsymbol\theta,\xi,|(\xi,\sigma)|)e^{-iR\sigma}|\sigma|.
$$
Hence, we have
$$
\begin{array}{ll}
||f||^2_\gamma&=\displaystyle\intL_{\RR^3}(1+|\boldsymbol\iota|^2+|\xi|^2)^\gamma|\hat{f}(\boldsymbol\iota,\xi)|^2d\boldsymbol\iota d\xi\\
&=\displaystyle\intL_{\RR^{}}\intL_{S^{1}}\intL\half |\sigma|^{}(1+|\sigma|^2+|\xi|^2)^\gamma|\hat{f}(\sigma\boldsymbol\theta,\xi)|^2d\sigma dS(\boldsymbol\theta) d\xi\\
&=\displaystyle\frac{16}{\pi^{2}}\intL_{S^{1}}\intL_{\RR^{}}\intL\half|\sigma|^{3}(1+|(\xi,\sigma)|^2)^\gamma|\tilde {g}(\boldsymbol\theta,\xi,|(\xi,\sigma)|)|^2d\sigma  d\xi dS(\boldsymbol\theta)\\
&\leq\displaystyle\frac{16}{\pi^{2}}\intL_{S^{1}}\intL_{\RR^{}}\intL\half(1+|(\xi,\sigma)|^2)^\gamma|(\xi,\sigma)|^{2}|\tilde {g}(\boldsymbol\theta,\xi,|(\xi,\sigma)|)|^2|\sigma|d\sigma d\xi dS(\boldsymbol\theta) \\
&=\displaystyle\frac{16}{\pi^{2}}\intL_{S^{1}}\intL_{\RR^{}}\intL^\infty_{|\xi|}(1+|\eta|^2)^\gamma|\eta|^{2}|\tilde {g}(\boldsymbol\theta,\xi,\eta)|^2\eta d\eta d\xi dS(\boldsymbol\theta) ,
\end{array}
$$
where in the last line, we changed the variable $|(\xi,\sigma)|$ to $\eta$.
\end{proof}

 %%%%%%%%%%%%%%%%%%%%%%%%%%%%%%%%%%%%%%%%%%%%%%%%%%%%%%%%
\subsection{Range conditions}\label{rangecylinder}
%%%%%%%%%%%%%%%%%%%%%%%%%%%%%%%%%%%%%%%%%
%From Theorem~\ref{thm:inversioncylinder}, we have some necessary range condition as following:
A range description is a collection of a priori conditions that the data $R_Cf$ must satisfy. 
Practically speaking, a range description enables us to check if the data we have is enough to produce an image, possible saving us the trouble of unnecessary computation. 
In this subsection, we describe the only necessary range conditions of the cylindrical Radon transform $R_Cf$.
\begin{theorem}\label{thm:necessaryrange}
If $g=R_Cf$ for $f\in C^\infty_c(B^2_R\times\RR)$, then we have
\begin{enumerate}
\item[1.] For any $z$ and $p$,
$$
\intR g(\boldsymbol\theta,p,\sqrt{(p-z)^2+(\rho-R)^2})dp=\intR g(-\boldsymbol\theta,p,\sqrt{(p-z)^2+(\rho+R)^2})dp.
$$
\item[2.] For $m=0,1,2,\ldots$, $\mathcal{P}_z(\boldsymbol\theta)$ is a homogeneous polynomial of degree $m$ in $\boldsymbol\theta$, where
$$
\mathcal{P}_z(\boldsymbol\theta)=\intR I^{-1}_\rho R^*_Cg(\boldsymbol\theta,z,R-s)s^mds.
$$
Here $I^{-1}_\rho h(\boldsymbol\theta, \zeta,\rho)$ is the Riesz potential defined by $\widehat {I^{-1}_\rho h}(\boldsymbol\theta,\xi,\sigma)=|\sigma|\hat h(\boldsymbol\theta,\xi,\sigma)$ for a function $h(\boldsymbol\theta,\zeta,\rho)$ on $S^1\times\RR\times\RR$ with its 2-dimensional Fourier transform $\hat h(\boldsymbol\theta,\xi,\sigma)$ with respect to $(\zeta,\rho)$.
%and
%\item supp $\tilde g\subset \{(\boldsymbol\theta,\xi,\eta)\in S^1\times\RR^2:\eta\geq |\xi|\}$.
\end{enumerate}
\end{theorem}
\begin{proof}%\indent
1. From equation~\eqref{eq:fourierslice}, $I^{-1}_\rho R^*_Cg(\boldsymbol\theta,z,\rho-R)$ should be equal to $I^{-1}_\rho R^*_Cg(-\boldsymbol\theta,z,-\rho-R)$.

2. This follows from equation~\eqref{relationfourierandbackcylinder} and the range condition of the regular Radon transform.
\end{proof}

  %%%%%%%%%%%%%%%%%%%%%%%%%%%%%%%%%%%%%%%%
 \section{An $n$-dimensional case of $R_C$}\label{sec:reconnd}
 %%%%%%%%%%%%%%%%%%%%%%%%%%%%%%%%%%%%%%%%%%%%%%
In this section, we consider the cylindrical Radon transform $R_C$ of a function $f\in C^\infty_c(B^k_R\times\RR^{n-k})$ where $n\geq 3$ is arbitrary.
As mentioned before (see also ~\cite{haltmeier11}), the cylindrical Radon transform $R_{C}$ of a 3-dimensional function $f\in C^\infty_c(B^2_R\times\RR^{})$ can be decomposed into the circular Radon transform and the regular $2$-dimensional Radon transform.
A natural $n$-dimensional analog of the cylindrical Radon transform would split into the composition of the $n-1$-dimensional spherical Radon transform and the 2-dimensional Radon transform. 
We consider a more general possibility.
Namely, we define $R_{C_{n,k}}f$ of a function $f\in C^\infty_c(B^k_R\times\RR^{n-k})$ that decomposes into the $n-k+1$-dimensional spherical Radon transform and the regular $k$-dimensional Radon transform.
We define $R_{C_{n,k}}f$ for $1<k\leq n-1$ and $(\boldsymbol\theta,\pp,r)\in S^{k-1}\times\RR^{n-k}\times[0,\infty)$ as follows: 
$$
R_{C_{n,k}}f(\boldsymbol\theta,\mathbf p,r)=\displaystyle\frac{1}{|S^{n-k}|}\intL_{\boldsymbol\theta^\perp}\int\limits_{S^{n-k}}f(\boldsymbol\tau+(R-r\alpha_1)\boldsymbol\theta, \pp+r\boldsymbol\alpha')dS(\boldsymbol\alpha) d\boldsymbol\tau,
$$
where $\boldsymbol\alpha=(\alpha_1,\boldsymbol\alpha')\in S^{n-k}$ and $\boldsymbol\theta^\perp$ refers to $\{(\boldsymbol\tau,\pp)\in \RR^{k}\times\RR^{n-k}:\boldsymbol\tau\cdot\boldsymbol\theta=0\}$.
Then we have an analogue of the Fourier slice theorem, similar to Theorem~\ref{thm:inversioncylinder3d}.

\begin{theorem}\label{thm:inversioncylindernd}
Let $f\in C^\infty_c(B^{k}_R\times \RR^{n-k})$. %, where $B^k_R$ is the ball in $\RR^k$ with radius $R$. 
If $g=R_{C_{n,k}}f$, then we have for $(\boldsymbol\theta,\boldsymbol\xi,\sigma)\in S^{k-1}\times \RR^{n-k}\times\RR$,
\begin{equation*}
\hat{f}(\sigma\boldsymbol\theta,\boldsymbol\xi)=2|S^{n-k}|(2\pi)^{-n+k-1}\widehat{R^*_{C_{n,k}}g}(\boldsymbol\theta,\boldsymbol\xi,\sigma)e^{-iR\sigma}|(\boldsymbol\xi,\sigma)|^{n-k-1}|\sigma|,
\end{equation*}
where {\color{black}$\hat f$ is the $n$-dimensional Fourier transform of $f$, i.e.,
$$
\hat{f}(\boldsymbol\xi)=\intL_{\RR^{n-k}}\intL_{\RR^k} f(\xx,\zz)e^{-i(\xx,\zz)\cdot \boldsymbol\xi} d\xx d\zz,\quad \boldsymbol\xi=(\xi_1,\xi_2,\cdots,\xi_n)\in\RR^n,
$$
%and $\widehat{R^*_Cg}$ is the $2$-dimensional Fourier transform of $R^*_Cg$ with respect to $(z,\rho)$, i.e.,
and $\widehat{R^*_{C_{n,k}}g}$ is the $n-k+1$-dimensional Fourier transform of $R^*_{C_{n,k}}g$ in $(\boldsymbol\zeta,\rho)$, i.e.,
$$
\widehat{R^*_{C_{n,k}}g}(\boldsymbol\theta,\boldsymbol\xi,\sigma)=\intR\intL_{\RR^{n-k}} R^*_Cg(\boldsymbol\theta,\boldsymbol\zeta,\rho)e^{-i(\boldsymbol\zeta,\rho)\cdot(\boldsymbol\xi,\sigma)}d\boldsymbol\zeta d\rho.
$$
Here 
$$
R^*_{C_{n,k}}g(\boldsymbol\theta,\boldsymbol\zeta,\rho)=\displaystyle\intL_{\RR^{n-k}} g(\boldsymbol\theta,\pp,\sqrt{|\boldsymbol\zeta-p|^2+\rho^2})dp,
$$
for $g\in C^\infty_c(S^{k}\times\RR^{n-k}\times[0,\infty))$ and $(\boldsymbol\zeta,\rho)\in\RR^{n-k}\times\RR$.}
\end{theorem}
This theorem is natural in view of Remark 1 in section~\ref{defiandworkcylinder}. %~\ref{rmk:fourierslicecylinder}
%\begin{proof}
The proof of this theorem is the almost same as that of Theorem~\ref{thm:inversioncylinder3d}.
Instead of taking the Hankel transform in $r$, one takes the radial Fourier transform, {\color{black}i.e.,
\begin{equation}\label{eq:radialfourier}
\eta^\frac{n-k-2}2\intL\half r^\frac{n-k}2 J_\frac{n-k-2}2(r\eta)\widehat{R_{C_{n,k}}f}(\boldsymbol\theta,\boldsymbol\xi,r)dr.
\end{equation}
(When $f(\xx)=f_0(|\xx|)$ for $\xx\in\RR^n$, the Fourier transform $\hat f$ of $f$ with respect to $\xx\in\RR^n$ becomes
\begin{equation}\label{eq:radialfourier1}
\hat f(\boldsymbol\xi)=(2\pi)^\frac n2\intL\half r^\frac{n}2 |\boldsymbol\xi|^\frac{2-n}2J_\frac{n-2}2(r|\boldsymbol\xi|)f_0(r)dr.
\end{equation}
We call the left hand side of equation~\eqref{eq:radialfourier1} without the constant $(2\pi)^\frac n2$ the radial Fourier transform.)}
Also, we need the identity
$$
\intL_{S^{n-1}}e^{i\boldsymbol\xi\cdot\boldsymbol\theta}dS(\boldsymbol\theta)=(2\pi)^{n/2}|\boldsymbol\xi|^{(2-n)/2}J_{(n-2)/2}(|\boldsymbol\xi|)
$$
(see~\cite{fawcett85}). Lastly instead of the identity~\eqref{eq:batemann2}, we need the following identity: for $a,b>0$ and $\mu>\nu>-1$,%for $0<\xi_1<\eta$,
\begin{equation}\label{eq:batemannn}
\begin{array}l
\displaystyle\int\limits\half \rho^{\nu+\frac12}(\rho^2+\beta^2)^{-\frac12\mu}J_\mu(a\sqrt{\rho^2+b^2})J_\nu(\rho|\boldsymbol\xi|)(\rho|\boldsymbol\xi|)^{1/2}d\rho\\
=\left\{\begin{array}{ll}a^{-\mu}|\boldsymbol\xi|^{\nu+\frac12}b^{-\mu+\nu+1}(a^2-|\boldsymbol\xi|^2)^{\frac12\mu-\frac12\nu-\frac12}J_{\mu-\nu-1}(b\sqrt{a^2-|\boldsymbol\xi|^2}) &\mbox{ if } 0<|\boldsymbol\xi|<a,\\\\
0&\mbox{ otherwise}\end{array}\right.
\end{array}
\end{equation}
\cite[p.59 (18) vol.2]{batemann}. %or for $n=2$ p.55 (35) vol.1
The other steps are the same as in the proof of Theorem~\ref{thm:inversioncylinder3d}.
%\end{proof}

For $\gamma<{n-k+1}$, we define the linear operators $I^\gamma$ and $I^\gamma_\rho$ by
$$
\widehat{I^\gamma h}(\boldsymbol\theta,\boldsymbol\xi,\sigma)=|(\boldsymbol\xi,\sigma)|^{-\gamma}\hat{h}(\boldsymbol\theta,\boldsymbol\xi,\sigma) \mbox{ and } \widehat{I^\gamma_\rho h}(\boldsymbol\theta,\boldsymbol\xi,\sigma)=|\sigma|^{-\gamma}\hat{h}(\boldsymbol\theta,\boldsymbol\xi,\sigma),
$$
for a smooth and compactly supported function $h(\boldsymbol\theta,\boldsymbol\zeta,\rho)$ on $S^{k-1}\times\RR^{n-k}\times\RR$ with its $n-k+1$-dimensional Fourier transform $\hat h$ with respect to $(\boldsymbol\zeta,\rho)$.
Then we have the inversion similar to Theorem~\ref{thm:3dcylinder}.
\begin{theorem}\label{thm:generalcylinder}
Let $f\in C^\infty_c(B^k_R\times\RR^{n-k})$. If $g=R_{C_{n,k}}f$, then we have for $(\xx,\zz)\in\RR^k\times\RR^{n-k},$
\begin{equation}\label{eq:inversioncylinern}
f(\xx,\zz)=\frac{|S^{n-k}|}{(2\pi)^{n}}\intL_{S^{k-1}}\left.I^{-k}_\rho I^{1-n+k}R^*_{C_{n,k}}g(\boldsymbol\theta,\zz,\rho)\right|_{\rho=\xx\cdot\boldsymbol\theta-R}dS(\boldsymbol\theta).
\end{equation}
\end{theorem}
To obtain inversion formula similar to Theorem~\ref{thm:cormack}, let $f(t,\boldsymbol\varphi,\zz)$ be the image function in cylindrical coordinates where $t=|\xx|$ and $\boldsymbol\varphi=\xx/|\xx|\in S^{k-1}$.
Let us expand $f(\rho,\boldsymbol\varphi,\zz)$ and $g(\boldsymbol\theta,\pp,r)$ in spherical harmonics:
%Then the spherical series of  with respect to their angular variables can be written as follows:
$$
f(t,\boldsymbol\varphi,\zz)=\displaystyle \sum^\infty_{l=0}\sum^{N(k,l)}_{j=0}f_{lj}(t,\zz)Y_{lj}(\boldsymbol\varphi) \quad\mbox{and}\quad g(\boldsymbol\theta,\pp,r)=\displaystyle \sum^\infty_{l=0}\sum^{N(k,l)}_{j=0}g_{lj}(\pp,r)Y_{lj}(\boldsymbol\theta),
$$
where $Y_{lj}$ is a spherical harmonic and
$$
N(k,l)=\frac{(2l+k-2)(k+l-3)!}{l!(k-2)!},\;\; N(k,0)=1.
$$

According to~\cite{natterer01}, when $g=\mathcal Rf$ for the regular $k$-dimensional Radon transform $\mathcal R$ and $g_{lj}$ and $f_{lj}$ are the spherical coefficients of $g$ and $f$, we have for $t>0$,
\begin{equation}\label{eq:sphericalharmonicsn}
f_{lj}(t)=\frac{(-1)^{k-1}}{2\pi^{k/2}}\frac{\Gamma((k-2)/2)}{\Gamma(k-2)}t^{2-k} \intL^\infty_t (s^2-t^2)^{(k-3)/2}C^{(k-2)/2}_l\left(\frac st\right)\frac{\partial^{k-1}}{\partial s^{k-1}}g_{lj}(s)ds,
\end{equation}
where $C_l^{(k-2)/2}$ is the (normalized) Gegenbauer polynomial of degree $l$.
From Theorem~\ref{thm:inversioncylindernd}, we have
\begin{equation}\label{eq:relationfourierandbackcylindernd}
%\begin{array}{ll}
\displaystyle \intL_{\boldsymbol\theta^\perp} \hat{f}(\boldsymbol\tau+s\boldsymbol\theta, \boldsymbol\xi)d\boldsymbol\tau \displaystyle\displaystyle=\frac{2|S^{n-k}|}{(2\pi)^{n-k+1}} \intR \widehat{R^*_{C_{n,k}}g}(\boldsymbol\theta,\boldsymbol \xi,\sigma) e^{i(s-R)\sigma}|\sigma| |(\boldsymbol\xi,\sigma)|^{n-k-1}d\sigma,
% \end{array}
\end{equation}
where $\hat{f}(\boldsymbol\tau+s\boldsymbol\theta, \boldsymbol\xi)$ is the $k$-dimensional Fourier transform of $f(\boldsymbol\tau+s\boldsymbol\theta, \zz)$ with respect to $\zz$.
Consider the $lj$-th spherical coefficient of the right hand side of formula~\eqref{eq:relationfourierandbackcylindernd}.
Then we have
\begin{equation}\label{eq:ljfouriercoefficient}
\begin{array}{l}
\displaystyle\intL_{S^{k-1}}\intR\widehat{R^*_{C_{n,k}}g}(\boldsymbol\theta,\boldsymbol\xi,\sigma)e^{i(s-R)\sigma}|\sigma| |(\boldsymbol\xi,\sigma)|^{n-k-1} Y_{lj}(\boldsymbol\theta) d\sigma dS(\boldsymbol\theta)\\
\displaystyle=\intL\half\widehat{R^*_{C_{n,k}}g_{lj}}(\boldsymbol\xi,\sigma)e^{i(s-R)\sigma}|\sigma| |(\boldsymbol\xi,\sigma)|^{n-k-1}d\sigma.
\end{array}
\end{equation}
%Combining equations~\eqref{eq:sphericalharmonics} and~\eqref{eq:ljfouriercoefficient}, we have the following theorem.
\begin{theorem}
Let $f\in C^\infty_c(B^k_R\times\RR^{n-k})$. Then we have for $\rho>0$
$$
f_{lj}(t,\zz)=2\pi c_kt^{2-k}\intL^\infty_t (s^2-t^2)^{\frac{k-3}2}C^{\frac{k-2}2}_l\left(\frac st\right) H_s\frac{\partial^{k}}{\partial s^{k}}I^{1+k-n}R^*_{C_{n,k}}g_{lj}(\zz,s-R)  ds,
$$
where 
$$
%\displaystyle c_k=\displaystyle\frac{2^\frac k2|S^{n-k}|}{2\pi^{\frac{n+k+1}2}}\frac{\Gamma((k-2)/2)}{\Gamma(k-2)}.
\displaystyle c_k=\displaystyle\frac{(-1)^{k-1}|S^{n-k}|}{\pi^\frac k2(2\pi)^{n-k+1}}\frac{\Gamma((k-2)/2)}{\Gamma(k-2)}.
$$
\end{theorem}
\begin{proof}
Applying equation~\eqref{eq:sphericalharmonicsn} to equations~\eqref{eq:relationfourierandbackcylindernd} and~\eqref{eq:ljfouriercoefficient} implies that $\widehat{f_{lj}}(t,\boldsymbol\xi)$ is equal to
\begin{equation*}
\begin{split}
\displaystyle c_kt^{2-k} \intL^\infty_t (s^2-t^2)^{\frac{k-3}2}C^{\frac{k-2}2}_l\left(\frac st\right) \intL\half i\operatorname{sgn}(\sigma)&{\left(\frac{\partial^{k}}{\partial s^{k}}R^*_{C_{n,k}}g_{lj}\right)}\widehat{\;}\;(\boldsymbol\xi,\sigma)e^{i(s-R)\sigma}\\
&\times|(\boldsymbol\xi,\sigma)|^{n-k-1} d\sigma ds,
\end{split}
\end{equation*}
where $\widehat{f_{lj}}$ is the $n-k$-dimensional Fourier transform of $f_{lj}$ with respect to $\zz$.
\end{proof}
Also, we can get the following theorem similar to Theorem~\ref{thm:reddingandradon}.
\begin{theorem}
  Let $f\in C^\infty_c(B^k_R\times\RR^{n-k})$.
  %If $u\geq R$, 
  Then $\int_{\boldsymbol\theta^\perp} f(\boldsymbol\tau+s\boldsymbol\theta,\zz)d\boldsymbol\tau$ can be represented as
\begin{equation*}
\displaystyle\frac{4|S^{n-k}|}{(2\pi)^{n-k+1}}\intR\intL_{\RR^{n-k}}\int\limits\half (R-s) rR_{C_{n,k}}f(\boldsymbol\theta,-\pp,r) e^{-ir^2\xi}e^{-i(2\zz\cdot \pp+(|\zz|^2+(R-s)^2)+|\pp|^2)\xi}\xi drd\pp d\xi.
  \end{equation*}
\end{theorem}

As in section~\ref{defiandworkcylinder}, we can obtain a stability estimate, a support theorem and necessary range conditions for $R_{C_{n,k}}$.
\begin{theorem}\label{rmk:norm}
We have
\begin{equation*}\label{eq:estimate}
||f||_{\gamma}\leq 2|S^{n-k}|(2\pi)^{-n+k-1}\ ||R_{C_{n,k}}f||_{\gamma+(n-1)/2},
\end{equation*}
for $f\in \mathcal H^{\gamma}(\RR^n)$ supported in $B^{k}_R\times\RR^{n-k}$.
\end{theorem}
%Also, we have a similar support theorem to Theorem~\ref{thm:support}.
\begin{theorem}
Let $\pp_0\in\RR^{n-k}$ and $B>0$.
Let $f\in C^\infty(B^k_R\times\RR^{n-k})$ and suppose that $g=R_{C_{n,k}}f$ is equal to zero on the open set $U_{B}=\{(\boldsymbol\theta,\pp_0,r):0\leq\boldsymbol\theta<2\pi,0\leq r<B\}$. Then $f$ is equal to zero on the set $\{(\xx,\zz)\in\RR^k\times\RR^{n-k}:|x|> R-\sqrt{B^2-|\zz-\pp_0|^2}\}$.
\end{theorem}
 \begin{remark}
We can obtain the same result as Theorem~\ref{thm:necessaryrange} for an $n$-dimensional case using Theorem~\ref{thm:inversioncylindernd} instead of Theorem~\ref{thm:inversioncylinder3d}.
 \end{remark}
%%%%%%%%%%%%%%%%%%%%%%%%%%%%%%%%%%%%%%%%%%%%%%%%%%%%%%%%%%%%%%%
\section{Planar geometry}\label{plane}
%%%%%%%%%%%%%%%%%%%%%%%%%%%%%%%%%%%%%%%%%%%%%%%%%
Let us first explain the mathematical model arising in PAT with line detectors introduced in~\cite{haltmeier09}.
Let $L_P(\boldsymbol\theta,p)=\{(0,y,z)\in\RR^3:(y,z)\cdot\boldsymbol\theta=p\}$ for $p>0$ and $ \boldsymbol\theta\in S^1$ be the location of a line detector.
Then we have $L_P(\boldsymbol\theta,p)=L_P(-\boldsymbol\theta,-p)$ and the detector line $L_P(\boldsymbol\theta,p)$ is located on the $yz$-plane.
\begin{definition}
Let a function $f$ be even in $x$.
The cylindrical Radon transform $R_P$ maps $f\in C^\infty_c(\RR^3)$ into 
 $$
 R_Pf(\boldsymbol\theta,p,r)=\displaystyle\frac{1}{2\pi r}\iint\limits_{d(L_P(\boldsymbol\theta,p),(x,y,z))=r} f(x,y,z)d\varpi,
 $$
 for $(\boldsymbol\theta,p,r)\in S^1\times\RR\times[0,\infty)$.
Here $d\varpi$ is the area measure on the cylinder 
$$
\{(x,y,z)\in\RR^3:d(L_P(\boldsymbol\theta,p),(x,y,z))=r\},
$$ 
and 
$$
d(L_P(\boldsymbol\theta,p),(x,y,z)):=\sqrt{x^2+(p-(y,z)\cdot\boldsymbol\theta)^2}
$$
 denotes the Euclidean distance between the line $L_P(\boldsymbol\theta,p)$ and the point $(x,y,z)$.
\end{definition}
We notice that if a function $f$ is odd in $x$, then $R_Pf$ is equal to zero.
\begin{remark}\label{rmk:evenness}
We have $R_Pf(\boldsymbol\theta,p,r)=R_Pf(-\boldsymbol\theta,-p,r)$.
\end{remark}
By definition, we have
$$
%\begin{array}{ll}
R_Pf(\boldsymbol\theta,p,r)=\displaystyle\frac{1}{2\pi}\intR\int\limits^{\pi}_{-\pi}f(r\cos\psi, t\boldsymbol\theta^\perp+(p-r\sin\psi)\boldsymbol\theta)d\psi dt,
%&=\displaystyle\frac{1}{\pi }\intR\int\limits^1_{-1}f(r\sqrt{1-s^2},t \boldsymbol\theta^\perp+(p-rs)\boldsymbol\theta)\frac{ds}{\sqrt{1-s^2}} dt,
%\end{array}
$$
where $r$ is the radius of the cylinder of integration, $p$ and $\boldsymbol\theta$ are the distance and the direction from the origin to the central axis of the cylinder, $t$ is a parameter variable of the central axis of the cylinder, and $\psi$ is the polar angle of the circle that is the intersection of a plane $\{(x,t\boldsymbol\theta):t\in\RR,x\in\RR\}$ and the cylinder.

\begin{figure}
\begin{center}
  \begin{tikzpicture}[>=stealth,scale=1]
    \draw[thin,->] (0,3) -- (0,4);
    \draw[thin,dotted] (0,1) -- (0,3);
    \draw[thin] (0,0) -- (0,1);
    \draw[thin,->] (3,0) -- (4,0);
    \draw[thin,dotted] (1,0) -- (3,0);
    \draw[thin] (0,0) -- (1,0);
    \draw[thin,->] (0,0) -- (-1.8,-1);
    \draw[densely dashed] (3.8,-1.8) -- (-1.5,3.5);%L_C(\boldsymbol\theta,p)
    \draw[very thick] (3.5,-0.5) -- (-0.5,3.5);
    \draw[very thick] (2.5,-1.5) -- (-1.5,2.5);
    \draw[very thick,rotate around={45:(3,-1)}] (3,-1) ellipse (1.4141/2 and 0.35) ;
    \draw[very thick,rotate around={45:(-1.5,2.5)},dashed] (-1.5,2.5) arc (180:360:1.4141/2 and 0.35) ;
    \draw[very thick,rotate around={45:(-.5,3.5)}] (-0.5,3.5) arc (0:180:1.4141/2 and 0.35) ;
    \draw[<->,rotate around={45:(0,0)},loosely dashed] (0,0) arc (180:360:1.4141/2 and 0.2) ;
    \draw[->] (0,0) -- (1,1);
    \node at (0.5,-2.8) {(a)};
    \node at (4.1,-2.2) {$L_P(\boldsymbol\theta,p)$};
    \node at (0.1,4.2) {$z$};
    \node at (0.7,1) {$\boldsymbol\theta$};
    \node at (0.5,0.1) {$p$};
    \node at (-1.5,-1.7) {$x$};
    \node at (4,0.2) {$y$};
    \draw[thin,->] (8,-2) -- (8,4);
    \draw[thin,->] (5.5,-1) -- (10.5,-1);
%    \draw[dotted] (5,-2.2) -- (5,4);    
%    \draw[densely dashed] (7,-2.2) -- (7,4);    
    \draw[<->,loosely dashed] (8,1) arc (90:270:0.4 and 1) ;  
 %   \draw[<->,loosely dashed] (8,1) arc (180:360:1.5/2 and 0.2) ;
    \draw[very thick] (8,1) circle (1.5);
    \draw[<->,loosely dashed] (9.5,1) -- (8,1);    
    \draw[densely dashed] (8+1.732*1.5/2,1+1.5/2) -- (8,1);    
    \draw[very thick,rotate around={-45:(8.4,1)}] (8.4,1) arc (50:80:0.4 and 0.4) ;
    \node at (8,-2.8) {(b)};
    \node at (8.2,3.7) {$\boldsymbol\theta$};
    \node at (8,-2.8) {(b)};
    \node at (7.5,.4) {$p$};
    \node at (8.5,0.7) {$r$};
    \node at (8.6,1.2) {$\psi$};
    \node at (10.5,-1.2) {$x$};
  \end{tikzpicture}
%        \subfigure[]{
%            \includegraphics[width=0.45\textwidth]{cylinderplane1.jpg}        }\qquad%
%          %\subfigure[]{
%           %  \includegraphics[width=0.3\textwidth]{cylinderzplane.eps}}
%                       \subfigure[]{
%             \includegraphics[width=0.45\textwidth]{cylinderonplane.jpg}
%             }

\end{center}
\caption{(a) the cylinder of integration whose the central axis is located on the $yz$-plane and (b) the restriction to the $\{(x,t\boldsymbol\theta):x\in\RR,t\in\RR\}$ plane}
\label{fig:cylinderplane}
\end{figure}
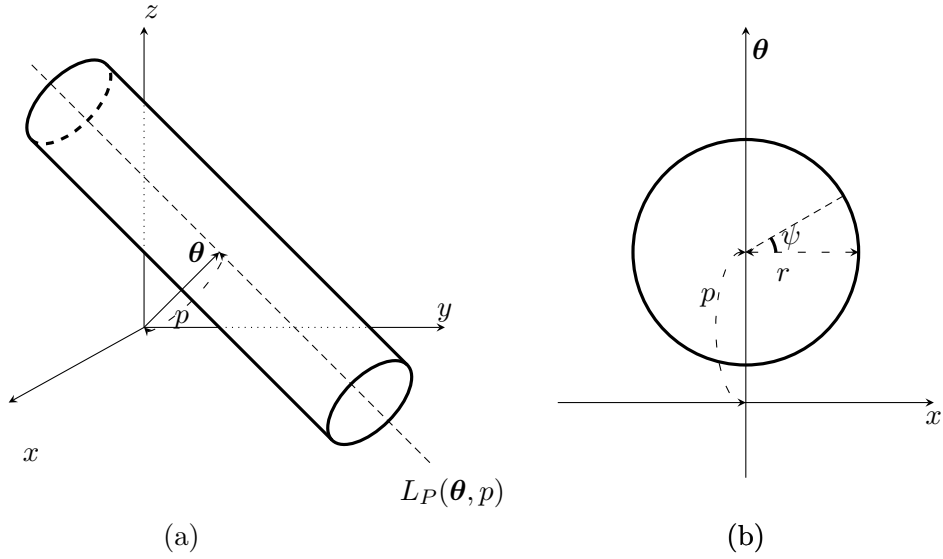

%%%%%%%%%%%%%%%%%%%%%%%%%%%%%%%%%%%%%%%%%%%%
\subsection{Inversion formulas}\label{reconplane}
%%%%%%%%%%%%%%%%%%%%
We have two integrals in the definition formula of $R_Pf$. 
Like $R_Cf$, the inner integral is a circular Radon transform with centers on the line for fixed $\boldsymbol\theta$, and the outer integral can be thought of as the $2$-dimensional regular Radon transform for a fixed $x$-coordinate~\cite{haltmeier09}.
Similarly, we start to apply the inversion of the circular Radon transform for a fixed $\boldsymbol\theta$.
%, if we, then we get the $2$-dimension regular Radon transform.

\begin{theorem}\label{thm:inversionplane}
Let $f\in C^\infty_c(\RR^3)$ be even in $x$.
If $g=R_Pf$, then we have
\begin{equation}\label{eq:invplane3d}
\hat{f}(\xi,\sigma \boldsymbol\theta) =4 |\xi| \widehat{R^*_Pg}(\boldsymbol\theta, \sigma,\xi),
\end{equation}
where $\hat f$ is the $3$-dimensional Fourier transform of $f$  with respect to $(x,y,z)\in\RR^3$ and $\widehat{R^*_Pg}$ is the 2-dimensional Fourier transforms of $R^*_Pg:=R^*_Cg$ and $(\zeta,\rho)\in\RR^2$ (see Theorem~\ref{thm:inversioncylinder3d}).
\end{theorem}
Notice that the evenness of $g$ in $(\boldsymbol\theta,p)$ implies the evenness of $\widehat{R^*_Pg}$ in $(\boldsymbol\theta,\sigma)$.
\begin{remark}\label{rmk:fouriersliceplane}
%Like $R_C$, $R_P$ can also be thought of as the composition of the circular Radon transform and the regular Radon transform.
Theorem~\ref{thm:inversionplane} can be thought of as the combination of two Fourier slice theorems for the circular and regular Radon transforms.
\end{remark}
\begin{proof}%[Proof of Theorem]
By definition, $R_Pf$ can be represented by
$$
R_Pf(\boldsymbol\theta,p,r)=\displaystyle\frac{1}{2\pi}\intR\int\limits^{1}_{-1}f(r\sqrt{1-s^2}, t\boldsymbol\theta^\perp+(p-rs)\boldsymbol\theta)\frac{ds}{\sqrt{1-s^2}} dt.
$$
Taking the Fourier transform of $R_Pf$ with respect to $p$ yields 
$$
\begin{array}{ll}
\widehat{R_Pf}(\boldsymbol\theta,\sigma,r)&\displaystyle=\frac{1}{\pi }\int\limits^1_{-1}\hat{f}(r\sqrt{1-s^2},\sigma \boldsymbol\theta) e^{irs\sigma}\frac{ds}{\sqrt{1-s^2}},
\end{array}
$$
where $\hat f$ and $\widehat{R_Pf}$ are the Fourier transforms of $f$ and $R_Pf$ with respect to $(y,z)\in\RR^{2}$ and $p\in\RR$, respectively.
Taking the Hankel transform $H_0$ of $\widehat{R_Pf}$ with respect to $r$, we have %for $\eta$%$|\xi|^2=\xi^2_1+\xi^2_2$,
\begin{equation}\label{eq:hankelrcf}
\begin{array}{ll}
H_0\widehat{R_Pf}(\boldsymbol\theta,\sigma,\eta)&=\displaystyle\frac{1}{\pi }\int\limits\half \int\limits^1_{-1}\hat{f}(r\sqrt{1-s^2},\sigma \boldsymbol\theta) e^{irs\sigma}\frac{ds}{\sqrt{1-s^2}}J_0(r\eta)rdr\\
   &=\displaystyle\frac{2}{\pi }\int\limits\half\int\limits^1_{-1}\hat{f}(r\sqrt{1-s^2},\sigma \boldsymbol\theta)J_0(r\eta)r \cos(rs\sigma)\frac{ds}{\sqrt{1-s^2}} dr\\
      &=\displaystyle\frac{1}{2\pi }\int\limits\half\int\limits\half\hat{f}(b,\sigma \boldsymbol\theta)\cos(\rho\sigma) J_0(\eta\sqrt{\rho^2+b^2})d\rho db,
  \end{array}
  \end{equation}
  where in the last line, we made a change of variables $(r,s)\rightarrow (\rho,b)$ where $r=\sqrt{\rho^2+b^2}$ and $s=\rho/\sqrt{\rho^2+b^2}$.
Applying the identity~\eqref{eq:batemann2} to equation~\eqref{eq:hankelrcf}, we get 
\begin{equation*}%\label{eq:relationhankelandfourier}
H_0\widehat{R_Pf}(\boldsymbol\theta,\sigma,\eta)=\left\{\begin{array}{ll}\displaystyle\frac{2}{\pi }\int\limits\half\hat{f}(b,\sigma \boldsymbol\theta)\dfrac{\cos(b\sqrt{\eta^2-\sigma^2})}{\sqrt{\eta^2-\sigma^2}}db\;&\mbox{ if } 0<\sigma<\eta,\\
0&\mbox{ otherwise.}\end{array}\right.
 \end{equation*}
Substituting $\eta=\sqrt{\xi^2+\sigma^2}$ yields
\begin{equation}\label{eq:relationhankelandfourier}
H_0\widehat{R_Pf}(\boldsymbol\theta,\sigma,|(\sigma,\xi)|)=\displaystyle\frac{2}{\pi }\int\limits\half\hat{f}(b,\sigma \boldsymbol\theta)\dfrac{\cos(b\xi)}{\xi}db=\displaystyle\frac{1}{\pi }\hat{f}(\xi,\sigma \boldsymbol\theta)|\xi|^{-1}.
 \end{equation}
%where $\hat{f}$ is the dimensional Fourier transform of $f$.

As in the proof of Theorem~\ref{thm:inversioncylinder3d}, we change the right side of equation~\eqref{eq:relationhankelandfourier} into a term containing the backprojection operator $R_P ^*$.
We have $\widehat{R^*_Pg}(\boldsymbol\theta, \sigma,\xi)=2\pi H_0\hat{g}(\boldsymbol\theta,\sigma,|(\sigma,\xi)| )$, so we get equation~\eqref{eq:invplane3d}. 
\end{proof}

Let the linear operator $I_\zeta$ and $I_\rho$ be defined by $\widehat{I^{-1}_\zeta h}(\boldsymbol\theta,\sigma,\xi)=|\sigma|\hat{h}(\boldsymbol\theta,\sigma,\xi)$ and $\widehat{I^{-1}_\rho h}(\boldsymbol\theta,\sigma,\xi)=|\xi|\hat{h}(\boldsymbol\theta,\sigma,\xi)$ for a smooth and compactly supported function $h(\boldsymbol\theta,\zeta ,\rho)$ on $S^1\times\RR\times\RR$ with its 2-dimensional Fourier transform $\hat h(\boldsymbol\theta,\sigma,\xi)$ with respect to $(\zeta ,\rho)$.
%where $\hat{h}$ is the 2-dimensional Fourier transform of $h$ in the last two dimensional variable.
Then we have the following inversion formula.
\begin{theorem}
 Let $f\in C^\infty_c(\RR^3)$ be even in $x$.
Then we have for $g=R_Pf$
$$
f(x,y,z)=4\pi^{-1}\intL_{S^{1}}I^{-1}_\zeta I^{-1}_\rho R^*_Pg(\boldsymbol\theta,\boldsymbol\theta\cdot (y,z),x)dS(\boldsymbol\theta).
$$
\end{theorem}
Notice that from equation~\eqref{eq:invplane3d}, we have
\begin{equation}\label{eq:radonplane}
\intR f(x,t\boldsymbol\theta^\perp+s\boldsymbol\theta)dt=4I^{-1}_\rho R^*_P(\boldsymbol\theta,s,x).
\end{equation}
%where $\widehat{R^*_P}$ is the 1-dimensional Fourier transform of $R^*_P$ with respect to $\rho$.
As in section \ref{defiandworkcylinder}.1, let $f(x,t,\varphi)$ be the image function in cylindrical coordinates where $(y,z)=t(\cos\varphi,\sin\varphi)$. 
Consider the $l$-th Fourier coefficient of the right hand side of formula~\eqref{eq:radonplane}.
Then we have
\begin{equation}\label{eq:relationfourierandbackplane}
\displaystyle\intL_{S^{1}}I^{-1}_\rho R^*_Pg(\boldsymbol\theta,s,x)  e^{-il\vartheta}dS(\boldsymbol\theta) \displaystyle=I^{-1}_\rho R^*_Pg_{l}(s,x),
\end{equation}
where $\boldsymbol\theta=(\cos\vartheta,\sin\vartheta)$.
Applying equation~\eqref{eq:sphericalharmonics} to equation~\eqref{eq:relationfourierandbackplane}, we have the following theorem similar to Theorem~\ref{thm:cormack}.
\begin{theorem}
Let $f\in C^\infty_c(\RR^{3})$ be even in $x$. Then we have for $t>0$
$$
f_{l}(x,t)=-\frac4\pi \intL^\infty_t (s^2-t^2)^{-1/2}\cos\left(l\arccos\left(\frac st\right)\right) \frac{\partial}{\partial s}I^{-1}_\rho R^*_Pg_{l}(s,x)  ds.
$$
\end{theorem}

Also, we have another relation between the Radon transform and $R_P$ similar to Theorem~\ref{thm:reddingandradon}.
\begin{theorem}
Let $f\in C^\infty_c(\RR^3)$ be even in $x$.
Then we have
$$
\intR f(x,t\boldsymbol\theta^\perp+z\boldsymbol\theta)dt=\frac{2}{\pi}\intR\intR\intL\half zr R_Pf(\boldsymbol\theta,-p,r)e^{-ir^2\sigma}e^{-i(2xp+(z^2+x^2)+p^2)\sigma}\sigma dr dp d\sigma.
$$
\end{theorem}
The proof is the same as that of Theorem~\ref{thm:reddingandradon} except for the obvious necessary changes.

%%%%%%%%%%%%%%%%%%%%%%%%%%%%%%%%%%%5
%%%%%%%%%%%%%%%%%%%%%%%%%%%%%%%%%%%%%%%%%%%%
\subsection{A stability estimate}\label{estiandrangeplane}
%%%%%%%%%%%%%%%%%%%%%%%%%%%%%%%%%%%%%%%%%%%%%%%%
In this subsection, we discuss the stability estimate of the cylindrical Radon transform $R_P$.

For $\gamma\geq0$, let us define $\mathcal H^{\gamma}_e(\RR^n):=\{f\in\mathcal H^{\gamma}(\RR^n):f$ is even in $x\}$, where $(x,\zz)\in\RR\times\RR^{n-1}$. 
As in subsection~\ref{estiandrangecylinder}, let $L^2_1(S^{n-2}\times\RR\times[0,\infty))$ be the set of functions $g$ on $S^{n-2}\times\RR\times[0,\infty)$ with
$$
||g||^2:=\intL_{S^{n-2}}\intR\intL\half |g(\boldsymbol\theta,p,r)|^2rdrdpdS(\boldsymbol\theta)<\infty.
$$
Then $L^2_1(S^{n-2}\times\RR\times[0,\infty))$ is a Hilbert space. Also, by the Plancherel formula, we have $||g||=(2\pi)^{-3}||\tilde g||$, where
$$
\tilde g(\boldsymbol\theta,\sigma,|\boldsymbol\zeta|):=\intRR\intR g(\boldsymbol\theta,p,|\mathbf w|)e^{-i(p,\mathbf w)\cdot(\sigma,\boldsymbol\zeta)}dpd\mathbf w.
$$
Let $H^\gamma(S^{n-2}\times\RR\times[0,\infty))$ be the set of functions $g\in L^2_1(S^{n-2}\times\RR\times[0,\infty))$ with $||g||_\gamma<\infty$, where 
$$
%\displaystyle||f||_{(1-n)/2}^2=\intL_{\RR^{n-k}}\intL_{\RR^{k}} (1+|(\iota,\xi)|^2)^{(1-n)/2}|\hat{f}(\iota,\xi)|^2d\iota d\xi \mbox{ and }
\displaystyle||g||^2_\gamma:=\intL_{S^{n-2}}\intL_{\RR}\intL\half |\tilde {g}(\boldsymbol\theta,\sigma,\eta)|^2(1+|\sigma|^2+|\eta|^2)^\gamma \eta d\eta d\sigma dS(\boldsymbol\theta).
$$

\begin{theorem}\label{lem:stability}
For $\gamma\geq0$, there exists a constant $c$ such that for $f\in \mathcal H^\gamma_e(\RR^3)$,
$$
||f||_\gamma\leq c||R_Pf||_{\gamma+1}.
$$
\end{theorem}

\begin{proof}
Let $g=R_Pf$.
Similar to equation~\eqref{eq:relationhankelandbackcylinder}, we have
\begin{equation}\label{eq:hatandtildeplane}
\widehat{R^*_Pg}(\boldsymbol\theta,\sigma,\xi)=\intR e^{-i\sigma p}\intR\intR e^{-i(\zeta,\rho)\cdot(\sigma,\xi)}g(\boldsymbol\theta,p,|(\zeta,\rho)|)d\zeta d\rho dp=\tilde g(\boldsymbol\theta,\sigma,|(\sigma,\xi)|).
\end{equation}
Combining this equation \eqref{eq:hatandtildeplane} and equation~\eqref{eq:invplane3d}, we have
$$
\hat f(\xi,\sigma\boldsymbol\theta)=4|\xi|\tilde g(\boldsymbol\theta,\sigma,|(\sigma,\xi)|).
$$
Hence, we have
$$
\begin{array}{ll}
||f||^2_\gamma&=\displaystyle\intL_{\RR^3}(1+|\iota|^2+|\xi|^2)^\gamma|\hat{f}(\xi,\iota)|^2d\iota d\xi\\
&=\displaystyle2^{-1}\intL_{S^{1}}\intL_{\RR}\intR |\sigma|^{}(1+|\sigma|^2+|\xi|^2)^\gamma|\hat{f}(\xi,\sigma\boldsymbol\theta)|^2 d\xi d\sigma dS(\boldsymbol\theta)\\
&=\displaystyle8\intL_{S^{1}}\intL_{\RR}\intR |\sigma|^{}(1+|(\sigma,\xi)|^2)^\gamma|\xi|^{2}|\tilde g(\boldsymbol\theta,\sigma,|(\sigma,\xi)|)|^2 d\xi d\sigma dS(\boldsymbol\theta)\\
&=\displaystyle16\intL_{S^{1}}\intL_{\RR}\intL\half |\sigma|^{}(1+|(\sigma,\xi)|^2)^\gamma|\xi|^{2}|\tilde g(\boldsymbol\theta,\sigma,|(\sigma,\xi)|)|^2 d\xi d\sigma dS(\boldsymbol\theta)\\
&=\displaystyle16\intL_{S^{1}}\intL_{\RR}\intL^\infty_{|\sigma|} \sqrt{\eta^2-\sigma^2}|\sigma|^{}(1+\eta^2)^\gamma|\tilde g(\boldsymbol\theta,\sigma,\eta)|^2 \eta d\eta d\sigma dS(\boldsymbol\theta),
\end{array}
$$
where in the last line, we changed the variable $|(\sigma,\xi)|$ to $\eta$.
Continuing the computation yields
$$
%\begin{array}{ll}
||f||^2_\gamma%&\leq\displaystyle c\intL_{S^{1}}\intL_{\RR}\intL^\infty_{|\sigma|} |\sigma|^{}(\rho^2-\sigma^2)^{1/2}(1+\rho^2)^\gamma|\tilde g(\boldsymbol\theta,\sigma,\rho)|^2 \rho d\rho d\sigma dS(\boldsymbol\theta)\\
\leq\displaystyle c\intL_{S^{1}}\intL_{\RR}\intL^\infty_0(1+\eta^2)^{\gamma+1}|\tilde g(\boldsymbol\theta,\sigma,\eta)|^2\eta d\eta d\sigma dS(\boldsymbol\theta).
%\end{array}
$$
\end{proof}

%%%%%%%%%%%%%%%%%%%%%%%%%%%%%%%%%%%%%%%%%%%%%%%5
\subsection{Range conditions}\label{rangeplane}
From Theorem~\ref{thm:inversionplane}, we have necessary range conditions for $R_p$ as follows:
\begin{theorem}\label{thm:rangeplane}
If $g=R_Pf$ for a function $f\in C^\infty(\RR^3)$ even in $x$, then we have
\begin{enumerate}
\item[1.] $g(\boldsymbol\theta,p,r)=g(-\boldsymbol\theta,-p,r)$ and
\item[2.] for $m=0,1,2,\ldots$, $\mathcal{P}_x(\boldsymbol\theta)$ is a homogeneous polynomial of degree $m$ in $\boldsymbol\theta$, where
$$
\mathcal{P}_x(\boldsymbol\theta)=\intR g(\boldsymbol\theta,p,\sqrt{(s-p)^2+x^2})s^mds.
$$
%and
%\item supp $\tilde g(\boldsymbol\theta,\sigma,\eta)\subset \{(\boldsymbol\theta,\sigma,\eta)\in S^1\times\RR^2:\eta\geq |\sigma|\}$.
\end{enumerate}
\end{theorem}

\begin{proof}\indent
1. This is shown in Remark 1.%~\ref{rmk:evenness}.

2. From equation~\eqref{eq:invplane3d} and the range description of the regular Radon transform, we have that for fixed $x$, the polynomial $\int_\RR I^{-1}_\rho R^*_Pg(\boldsymbol\theta,s,x)s^mds$ is homogeneous of degree $m$ in $\boldsymbol\theta$,
which implies that $\mathcal{P}_x(\boldsymbol\theta)$ is a homogeneous polynomial of degree $m$ in $\boldsymbol\theta$.
%3. We can prove this proof in way similar to that of 3 in Theorem~\ref{thm:necessaryrange}.
%3. Let a function $\phi$ on $S^1\times\RR\times[0,\infty)$ satisfy supp $\tilde\phi(\boldsymbol\theta,\sigma,\eta)\subset \{(\boldsymbol\theta,\sigma,\eta):0\leq \eta\leq |\sigma|\}$.
%Then we have
%$$
%\begin{array}{ll}
%\displaystyle\intL_{S^1}\intR\intL\half R_Pf(\boldsymbol\theta,p,r)\phi(\boldsymbol\theta,p,r)r drdpdS(\boldsymbol\theta)&=\displaystyle\intL_{S^1}\intR\intR \intR f(x,t\boldsymbol\theta^\perp+s\boldsymbol\theta)dt R^*_P\phi(\boldsymbol\theta,s,x) dsdxdS(\boldsymbol\theta)\\
%&=\displaystyle\intL_{S^1}\intR\intR \hat f(\xi,\sigma\boldsymbol\theta)\widehat{ R^*_P\phi}(\boldsymbol\theta,\sigma,\xi) d\sigma d\xi dS(\boldsymbol\theta)\\
%&=\displaystyle\intL_{S^1}\intR\intR \hat f(\xi,\sigma\boldsymbol\theta)\tilde{ \phi}(\boldsymbol\theta,\sigma,|(\xi,\sigma)|) d\sigma d\xi dS(\boldsymbol\theta)=0.
%\end{array}
%$$
\end{proof}

%%%%%%%%%%%%%%%%%%%%%%%%%5
\section{An $n$-dimensional case of $R_P$}\label{planendimension}
%%%%%%%%%%%%%%%%%%%%%%%%%%%%%%%%%%
As in section~\ref{sec:reconnd},  we consider the cylindrical Radon transform $R_P$ of a function $f\in C^\infty_c(\RR^n)$.
Assume $n\geq 3$. We define $R_{P_n}$ of a function $f\in C^\infty_c(\RR^n)$ even in $x\in\RR$ by 
$$
R_{P_n}f(\boldsymbol\theta,p,r)=\displaystyle\frac{1}{2\pi}\intL_{\boldsymbol\theta^\perp}\int\limits^{2\pi}_0f(r\cos\psi,\boldsymbol\tau+(p-r\sin\psi)\boldsymbol\theta)d\psi d\boldsymbol\tau,
$$
for $(\boldsymbol\theta,p,r)\in S^{n-2}\times\RR\times[0,\infty)$ and $(x,\zz)\in\RR\times\RR^{n-1}$.
Here $\boldsymbol\theta^\perp$ actually refers to $\{(0,\boldsymbol\tau)\in\RR\times\RR^{n-1}:\boldsymbol\tau\cdot\boldsymbol\theta=0\}$. %$\boldsymbol\theta^\perp\cap\{(0,z):z\in\RR^{n-1}\}$.
We still have $R_{P_n}f(\boldsymbol\theta,p,r)=R_{P_n}f(-\boldsymbol\theta,-p,r)$.
The $n$-dimensional cylindrical Radon transform $R_{P_n}$ can be decomposed into the circular Radon transform and the regular $n-1$-dimensional Radon transform.

\begin{theorem}%\label{thm:inversionplanen}
Let $f\in C^\infty_c(\RR^n)$ be even in $x\in\RR$.
If $g=R_{P_n}f$, then we have
\begin{equation}\label{eq:invplanend}
\hat{f}(\xi,\sigma \boldsymbol\theta) =4 |\xi| \widehat{R^*_{P_n}g}(\boldsymbol\theta, \sigma,\xi),
\end{equation}
where $\hat f$ is the $n$-dimensional Fourier transform of $f$ with respect to $(x,\zz)\in\RR\times\RR^{n-1}$ and $\widehat{R^*_{P_n}g}$ is the 2-dimensional Fourier transform of with respect to $(\zeta,\rho)\in\RR\times\RR$.
Here for a function $g$ on $S^{n-2}\times\RR\times[0,\infty)$,
$$
R^*_{P_n}g(\boldsymbol\theta,\zeta,\rho)=\displaystyle\intR g(\boldsymbol\theta,p,\sqrt{(\zeta-p)^2+\rho^2})dp.
$$
\end{theorem}
This proof is similar to that of Theorem~\ref{thm:inversionplane}. 
%The only difference is that one takes the radial Fourier transform~\eqref{eq:radialfourier} and use equation~\eqref{eq:batemannn} instead of the Hankel transform and identity~\eqref{eq:batemann2} as in the proof of Theorem~\ref{thm:inversioncylindernd}.
%For $\gamma<{n-k+1}$, we define the linear operators $I^\gamma_$ and $I^\gamma_\rho$ by
%$$
%\widehat{I^\gamma h}(\boldsymbol\theta,\xi,\sigma)=|(\xi,\sigma)|^{-\gamma}\hat{h}(\boldsymbol\theta,\xi,\sigma) \mbox{ and } \widehat{I^\gamma_\rho h}(\boldsymbol\theta,\xi,\sigma)=|\sigma|^{-\gamma}\hat{h}(\boldsymbol\theta,\xi,\sigma),
%$$
%for a smooth and compactly supported function $h(\boldsymbol\theta,z,\rho)$ on $S^{k-1}\times\RR^{n-k}\times\RR$ with its $n-k+1$-dimensional Fourier transform $\hat h$ with respect to $(z,\rho)$.
\begin{theorem}
 Let $f\in C^\infty_c(\RR^n)$ be even in $x$.
Then we have 
$$
f(x,\zz)=2(2\pi)^{2-n}\intL_{S^{n-2}}I^{2-n}_\zeta I^{-1}_\rho R^*_{P_n}g(\boldsymbol\theta,\boldsymbol\theta\cdot \zz,x)dS(\boldsymbol\theta),
$$
for $g=R_{P_n}f$ and $( x,\zz)\in\RR\times\RR^{n-1}$.
\end{theorem}
Let $f(x,t,\boldsymbol\varphi)$ be the image function, where $t=|\zz|$ and $\boldsymbol\varphi=\zz/|\zz|\in S^{n-2}$. %Then the $f_{lj}(x,r)$ and $g_{lj}$ be Fourier coefficients of $f$ and $g$.
Consider the series of $f(x,t,\boldsymbol\varphi)$ and $g(\boldsymbol\theta,p,r)$ expanded in spherical harmonics:
%with respect to their angular variables can be written as follows:
$$
f(x,t,\boldsymbol\varphi)=\displaystyle \sum^\infty_{l=0}\sum^{N(n-1,l)}_{j=0}f_{lj}(x,t)Y_{lj}(\boldsymbol\varphi) \quad\mbox{and}\quad g(\boldsymbol\theta,p,r)=\displaystyle \sum^\infty_{l=0}\sum^{N(n-1,l)}_{j=0}g_{lj}(p,r)Y_{lj}(\boldsymbol\theta).
$$
From equation~\eqref{eq:invplanend}, we have
\begin{equation}\label{eq:radonplanen}
\intL_{\boldsymbol\theta^\perp}f(x,\boldsymbol\tau+s\boldsymbol\theta)d\boldsymbol\tau=4I^{-1}_\rho R^*_{P_n}g(\boldsymbol\theta,s,x).
\end{equation}
Consider the $lj$-th spherical coefficient of the right hand side of formula~\eqref{eq:radonplanen}.
Then we have
\begin{equation}\label{eq:relationfourierandbackplanen}
\displaystyle\intL_{S^{n-2}}I^{-1}_\rho R^*_{P_n}g(\boldsymbol\theta,s,x)  Y_{lj}(\boldsymbol\theta)  dS(\boldsymbol\theta)\displaystyle=I^{-1}_\rho R^*_{P_n}g_{lj}(s,x).
\end{equation}
Applying equation~\eqref{eq:relationfourierandbackplanen} to equation~\eqref{eq:sphericalharmonics}, we have the following theorem.
\begin{theorem}
Let $f\in C^\infty_c(\RR^{n})$ be even in $x$. Then we have for $t>0$,
$$
f_{lj}(x,t)=4c_{n-1}t^{3-n}\intL^\infty_t (s^2-t^2)^{\frac{n-4}2}C^{\frac{n-3}2}_l\left(\frac st\right) \frac{\partial^{n-2}}{\partial s^{n-2}}I^{-1}_\rho R^*_{P_n}g_{lj}(s,x)  ds,
$$
where 
$$
\displaystyle c_n=\displaystyle\frac{(-1)^{n-1}}{2\pi^{\frac{n}2}}\frac{\Gamma((n-2)/2)}{\Gamma(n-2)}.
$$
\end{theorem}

Also, as in subsection~\ref{reconplane}, we have the following theorem.
\begin{theorem}
Let $f\in C^\infty_c(\RR^n)$ be even in $x$.
Then we have
$$
\intL_{\boldsymbol\theta^\perp}f(x,\boldsymbol\tau+s\boldsymbol\theta)d\boldsymbol\tau=\frac{2}{\pi}\intR\intR\intL\half sr R_{P_n}f(\boldsymbol\theta,-p,r)e^{-ir^2\sigma}e^{-i(2xp+(s^2+x^2)+p^2)\sigma}\sigma dr dp d\sigma.
$$
\end{theorem}
As in section~\ref{plane}, we can obtain a stability estimate and necessary range conditions for $R_{P_n}$.
\begin{theorem}%\label{lem:stability}
For $\gamma\geq0$, there exists a constant $c_n$ such that for $f\in \mathcal H^\gamma_e(\RR^n)$,
$$
||f||_\gamma\leq c_n||R_{P_n}f||_{\gamma+n-2}.
$$
\end{theorem}
\begin{proof}
Let $g=R_Pf$.
As in the proof of Theorem~\ref{lem:stability}, using equation~\eqref{eq:invplanend}, we have  
$$
\hat f(\xi,\sigma\boldsymbol\theta)=4|\xi|\tilde g(\boldsymbol\theta,\sigma,|(\sigma,\xi)|),
$$
and
$$
\begin{array}{ll}
||f||^2_\gamma&=\displaystyle\intR\intL_{\RR^{n-1}}(1+|\boldsymbol\iota|^2+|\xi|^2)^\gamma|\hat{f}(\xi,\boldsymbol\iota)|^2d\boldsymbol\iota d\xi\\
%&=\displaystyle2^{-1}\intL_{S^{n-2}}\intL_{\RR}\intR |\sigma|^{n-2}(1+|\sigma|^2+|\xi|^2)^\gamma|\hat{f}(\xi,\sigma\boldsymbol\theta)|^2 d\xi d\sigma dS(\boldsymbol\theta)\\
&=\displaystyle8\intL_{S^{n-2}}\intL_{\RR}\intR |\sigma|^{n-2}(1+|(\sigma,\xi)|^2)^\gamma|\xi|^{2}|\tilde g(\boldsymbol\theta,\sigma,|(\sigma,\xi)|)|^2 d\xi d\sigma dS(\boldsymbol\theta)\\
%&=\displaystyle16\intL_{S^{1}}\intL_{\RR}\intL\half |\sigma|^{n-2}(1+|(\sigma,\xi)|^2)^\gamma|\xi|^{2}|\tilde g(\boldsymbol\theta,\sigma,|(\sigma,\xi)|)|^2 d\xi d\sigma dS(\boldsymbol\theta)\\
&=\displaystyle16\intL_{S^{n-2}}\intL_{\RR}\intL^\infty_{|\sigma|} \sqrt{\eta^2-\sigma^2}|\sigma|^{n-2}(1+\eta^2)^\gamma|\tilde g(\boldsymbol\theta,\sigma,\eta)|^2 \eta d\eta d\sigma dS(\boldsymbol\theta),
\end{array}
$$
Here, we changed the variables $|(\sigma,\xi)|$ to $\eta$.
Hence, we have
$$
\begin{array}{ll}
||f||^2_\gamma&\leq\displaystyle c_n\intL_{S^{n-2}}\intL_{\RR}\intL^\infty_{|\sigma|} |\sigma|^{n-2}(\eta^2-\sigma^2)^{(n-2)/2}(1+\eta^2)^\gamma|\tilde g(\boldsymbol\theta,\sigma,\eta)|^2 \eta d\eta d\sigma dS(\boldsymbol\theta)\\
&\leq\displaystyle c_n\intL_{S^{n-2}}\intL_{\RR}\intL^\infty_0(1+\eta^2)^{\gamma+n-2}|\tilde g(\boldsymbol\theta,\sigma,\eta)|^2\eta d\eta d\sigma dS(\boldsymbol\theta),
\end{array}
$$
since $2|\sigma|(\eta^2-\sigma^2)^{1/2}\leq \eta^2$.
\end{proof}

\begin{remark}
Theorem~\ref{thm:rangeplane} holds for $R_{P_n}$ for $n\geq 3$.
\end{remark}
%%%%%%%%%%%%%%%%%%%%%%%%%%%%%%%%%%%%
\section{Conclusion}
%%%%%%%%%%%%%%%%%%%%%%%%%%%%%%%%
In this article, we study two different versions of cylindrical Radon transforms arising in PAT. 
We describe some inversion formulas of these transforms and discuss their  stability estimate and necessary range conditions. %.are discussed.%: The first inversion obtained through taking Fourier transform with respect to the height of cylinder and then Hankel transform with respect to radius of cylinder is related to~\cite{andersson88,fawcett85,nattererw01,nilsson97} and the second inversion obtained through taking the modified Fourier transform with respect to radius of cylinder related to~\cite{reddingn01}.
\section*{Acknowledgements}
The author thanks P Kuchment and D Steinhauer for fruitful discussions.
We are also thankful to the referees for many suggestions that helped to improve this paper.
This work was supported in part by US NSF Grants DMS 0908208 and DMS 1211463.

\bibliographystyle{plain}
%\begin{thebibliography}{10}
%\bibliography{s.moonref}

\end{document}